\documentclass[12pt]{article}
\usepackage{amsmath}
\usepackage{amssymb}
\usepackage{amsthm}
\usepackage{a4wide}
\usepackage{lmodern}
\usepackage{microtype}
\usepackage{mathtools}
\usepackage[colorlinks=true,linkcolor=blue,citecolor=blue]{hyperref}
\usepackage{cleveref}
\crefformat{equation}{(#2#1#3)}
\Crefformat{equation}{(#2#1#3)}
\crefrangeformat{equation}{(#3#1#4)--(#5#2#6)}
\Crefrangeformat{equation}{(#3#1#4)--(#5#2#6)}
\crefmultiformat{equation}{(#2#1#3)}{ and (#2#1#3)}{, (#2#1#3)}{ and (#2#1#3)}
\Crefmultiformat{equation}{(#2#1#3)}{ and (#2#1#3)}{, (#2#1#3)}{ and (#2#1#3)}
\crefname{enumi}{part}{parts}
\Crefname{enumi}{Part}{Parts}
\creflabelformat{enumi}{(#2#1#3)}
\crefname{section}{Section}{Sections}
\Crefname{section}{Section}{Sections}
\crefname{subsection}{Section}{Sections}
\Crefname{subsection}{Section}{Sections}
\crefname{theorem}{Theorem}{Theorems}
\Crefname{theorem}{Theorem}{Theorems}
\crefname{proposition}{Proposition}{Propositions}
\Crefname{proposition}{Proposition}{Propositions}
\crefname{lemma}{Lemma}{Lemmas}
\Crefname{lemma}{Lemma}{Lemmas}
\crefname{corollary}{Corollary}{Corollaries}
\Crefname{corollary}{Corollary}{Corollaries}
\crefname{definition}{Definition}{Definitions}
\Crefname{definition}{Definition}{Definitions}
\crefname{remark}{Remark}{Remarks}
\Crefname{remark}{Remark}{Remarks}
\crefname{example}{Example}{Examples}
\Crefname{example}{Example}{Examples}
\usepackage{amsrefs}

\newtheorem{theorem}{Theorem}[section]
\newtheorem{proposition}[theorem]{Proposition}
\newtheorem{lemma}[theorem]{Lemma}
\newtheorem{corollary}[theorem]{Corollary}
\newtheorem{definition}[theorem]{Definition}
\newtheorem{remark}[theorem]{Remark}
\newtheorem{example}[theorem]{Example}
\newtheorem{conjecture}[theorem]{Conjecture}

\usepackage{etoolbox}
\newcommand{\creffixtype}[1]{\AtBeginEnvironment{#1}{\crefalias{theorem}{#1}}}
\creffixtype{theorem}
\creffixtype{proposition}
\creffixtype{lemma}
\creffixtype{corollary}
\creffixtype{definition}
\creffixtype{remark}
\creffixtype{example}
\creffixtype{conjecture}
\crefname{conjecture}{Conjecture}{Conjectures}
\Crefname{conjecture}{Conjecture}{Conjectures}

\newcommand{\Z}{\mathbb{Z}}
\newcommand{\Lieg}{\mathfrak{g}}
\newcommand{\Liet}{\mathfrak{t}}
\newcommand{\Lie}{\operatorname{Lie}}
\newcommand{\uu}{\mathfrak{u}}
\newcommand{\nn}{\mathfrak{n}}
\newcommand{\U}{\mathcal{U}}
\newcommand{\V}{\mathcal{V}}
\newcommand{\rk}{\operatorname{rank}}
\newcommand{\Hom}{\operatorname{Hom}}
\newcommand{\Ext}{\operatorname{Ext}}
\newcommand{\End}{\operatorname{End}}
\newcommand{\Aut}{\operatorname{Aut}}
\newcommand{\Lin}{\operatorname{Lin}}
\newcommand{\Mat}{\operatorname{Mat}}
\newcommand{\GL}{\operatorname{GL}}
\newcommand{\tr}{\operatorname{tr}}
\newcommand{\dimv}{\operatorname{\underline{dim}}}
\newcommand{\supp}{\operatorname{supp}}
\newcommand{\rest}{\big |}
\newcommand{\rad}[1]{#1_{\operatorname{unip}}}
\newcommand{\orbit}{\mathcal{O}}
\newcommand{\Sat}{\operatorname{Sat}}
\newcommand{\pSat}{\operatorname{p-Sat}}
\newcommand{\chr}{\operatorname{char}}
\newcommand{\pr}{\operatorname{pr}}
\newcommand{\diag}{\mathrm{d}}
\newcommand{\charexp}{\operatorname{p}}
\newcommand{\grph}{\mathcal{G}}
\newcommand{\Multgrp}{\mathbb{G}_{\operatorname{m}}}
\newcommand{\Addgrp}{\mathbb{G}_{\operatorname{a}}}
\newcommand{\IM}{\Delta}
\newcommand{\coinv}{\operatorname{coinv}}
\newcommand{\homr}{\vdash}
\newcommand{\homrt}{\Vdash}
\newcommand{\extr}{\prec}
\newcommand{\PVS}{PVS}

\title{A greedy open-orbit criterion for solvable algebraic group actions, with applications to Lusztig's nilpotent varieties}

\author{Erez Lapid\\
Department of Mathematics\\
Weizmann Institute of Science\\
Rehovot, Israel\\
\texttt{erez.lapid@weizmann.ac.il}}
\date{August 2026}

\begin{document}
\maketitle

\begin{abstract}
Let $G$ be a connected solvable algebraic group acting rationally on a finite-dimensional vector space $U$.
Using a $G$-stable complete flag, we formulate a successive-quotient procedure that decides whether
$U$ contains an open $G$-orbit. When the procedure succeeds, it constructs an open-orbit vector of minimum
support cardinality and determines the image of a generic stabilizer in the maximal torus quotient.
We also give an infinitesimal version detecting open separable orbits. We apply the criterion to the action of $\Aut_Q(M)$
on $\Ext^1_Q(M,M)^*$ where $M$ is a multiplicity-free representation of a Dynkin quiver.
Rigidity of the corresponding component of Lusztig's nilpotent variety is thereby reduced to a rank test together with an acyclicity condition on a graph of active extension coordinates;
the connected components of the resulting forest determine the generic indecomposable decomposition.
For equioriented type $A$ this yields an explicit algorithm for a family of multisegments encoded by incidence matrices, including nonregular examples with repeated beginnings or ends.
\end{abstract}

\setcounter{tocdepth}{1}
\tableofcontents

\section{Introduction}\label{sec:intro}
Prehomogeneous vector spaces (\PVS s) are rational representations of algebraic groups that admit an open (i.e., dense) orbit.
They arise throughout the study of nilpotent orbits.

\PVS s for reductive groups are well studied (\cites{MR430336, MR1944442}).
The purpose of this paper is twofold. First, we give a uniform successive-quotient criterion for the existence of an open orbit
in an arbitrary rational representation of a connected solvable algebraic group.
Then, we apply it in the context of Lusztig's nilpotent varieties for quivers of Dynkin type to get a combinatorial/linear algebraic algorithm
for rigidity of irreducible components in the multiplicity-free case. We explicate it further in type $A$.

Let $G$ be a connected solvable algebraic group over an algebraically closed field $K$ of any characteristic and
let $V$ be a rational $d$-dimensional representation of $G$ over $K$.
If $G$ is a torus, $V$ is a {\PVS} if and only if the weights of $V$ are $\Z$-independent characters.
In general, it turns out that there is a simple greedy procedure to determine whether $V$ admits an open orbit.
At the same time it provides a minimal-cardinality support set $\Psi\subset X^*(G)$ for vectors in an open orbit.

To describe the procedure, choose an invariant complete flag $V=V_0\supset V_1\supset\dots\supset V_d=0$ for $V$ (which exists by the Lie--Kolchin theorem).
Fix a maximal torus $T$ of $G$ and a basis $e_1,\dots,e_d$ of $T$-eigenvectors, with characters $\chi_1,\dots,\chi_d\in X^*(T)=X^*(G)$, such that $e_i\in V_{i-1}\setminus V_i$.
In the $k$-th step we are given a subset $\Phi_{k-1}\subset\{1,\dots,k-1\}$ such that the $G$-orbit of any vector $\sum_{i=1}^{k-1}\lambda_i e_i$
with support set $\supp v:=\{i:\lambda_i\ne0\}$ equal to $\Phi_{k-1}$ is dense modulo $V_{k-1}$.
If this orbit is dense modulo $V_k$ then we let $\Phi_k=\Phi_{k-1}$ and continue.
Otherwise, we set $\Phi_k=\Phi_{k-1}\cup\{k\}$, and continue if the set $\{\chi_i: i\in\Phi_k\}\subset X^*(T)$ is $\Z$-independent and abort otherwise.

If all steps $k=1,\dots,d$ go through then the set of characters $\Psi=\{\chi_i: i\in\Phi\}$ is $\Z$-independent and any vector $v$
with $\supp v=\Phi_d$ has an open orbit, while no vector of smaller size support has an open orbit.
In contrast, if the procedure aborts at $k$ then $V$, and in fact $V/V_k$ is not a \PVS.

In the case where a dense orbit exists, the set $\Psi$ may depend on the choice of flag.
However, the $\charexp$-saturation in $X^*(T)$ of $\langle\Psi\rangle$, and therefore the size of $\Psi$, are independent of the flag.
Moreover, the joint kernel of $\Psi$ in $T$ is a maximal diagonalizable subgroup of the stabilizer of any $v$ with $\supp v=\Phi$.

\begin{theorem}
Let $G$ be a connected solvable algebraic group acting rationally on a finite-dimensional vector space $U$, and fix a $G$-stable complete flag.
The successive-quotient procedure either:
\begin{enumerate}
\item terminates successfully and produces a vector with open $G$-orbit, or
\item terminates at a quotient $U_k$ on which $G$ has no open orbit.
\end{enumerate}
In the successful case, the number of active coordinates is the minimum support cardinality of an open-orbit vector.
It is also equal to the codimension of a generic orbit of $\rad{G}$.
The active characters determine the image of a generic stabilizer in the maximal torus quotient of $G$ and the associated
$\charexp$-saturated lattice is independent of the chosen flag.
\end{theorem}

We also consider the infinitesimal version, to detect separable open orbits. This is particularly useful since then the algorithm
becomes a linear algebra test.

The use of successive one-dimensional quotients is closely related to the
orbit algorithms of B\"urgstein--Hesselink \cite{MR879187} and Goodwin \cite{MR2123984} for Borel subgroups acting on
submodules of nilradicals.  Dense-orbit and finite-orbit questions for related Borel and
parabolic adjoint actions were studied by Hille--R\"ohrle \cite{MR1669178},
Br\"ustle--Hille--R\"ohrle--Zwara \cite{MR1736958}, and Goodwin--Hille \cite{MR2356319}, among others, often through the representation
theory of auxiliary quasi-hereditary algebras, the relevant algebra being the one attached by
Hille--Vossieck \cite{MR2033108} to the radical bimodule of a hereditary algebra.
Closest in spirit to the present paper, Jensen--Su--Yu \cite{MR2481983} construct a Richardson
element for every seaweed subalgebra of $\mathfrak{gl}_n$ by an inductive procedure, and exhibit a seaweed
subalgebra of a simple Lie algebra of type $E_8$ with no dense orbit at all; their argument runs
through the quasi-hereditary algebra of \cite{MR2033108} and not through the preprojective algebra.
Jensen and Su also considered the distinct but
related conjugation action of $\Aut_{KQ}(P)$ on $\operatorname{rad}\End_{KQ}(P)$ for a projective quiver
module $P$, and characterized Dynkin quivers by a uniform dense-orbit property \cite{MR3228477}.  Our
purpose here is to isolate the flagwise successive-quotient argument for an arbitrary
rational representation of a connected solvable algebraic group.  Besides deciding
prehomogeneity, the procedure constructs an open-orbit vector with minimum coordinate
support and determines the image of a generic stabilizer in the maximal torus quotient.

\Cref{sec:setup} proves the group-level criterion, including its
positive-characteristic and infinitesimal versions.  In \cref{sec:quiver} we apply it to the
solvable action of $\Aut_Q(M)$ on $\Ext^1_Q(M,M)^*$ for a multiplicity-free Dynkin-quiver
module $M$ and obtain a rank test together with a forest condition on active extension
coordinates.  \Cref{sec: typeA} specializes to equioriented type $A$ and gives an explicit
algorithm for multisegments encoded by incidence matrices, including examples with
repeated beginnings or ends. \Cref{sec:appendix} treats permutation matrices: the known
rigidity--smoothness equivalence is strengthened to strong rigidity, and the components of
the active forest are related to the generic decomposition into indecomposable
preprojective modules.

The application in \cref{sec:quiver} is our main motivation. 
The irreducible components of $\Lambda(V)$ canonically parametrize the canonical-basis elements \cite{MR1088333}, as well as 
the semicanonical basis \cite{MR1758244} and the crystal $B(\infty)$ \cite{MR1458969}.
Rigid components are those which contain an open orbit. They became important through the works of Geiss, Leclerc, Schr\"oer and others \cite{MR2115084}, \cite{MR2144987}.
The corresponding real simple modules over quiver Hecke (KLR) algebras play a prominent role in monoidal categorification of cluster algebras \cite{MR3758148}.
Leclerc's imaginary vectors \cite{MR1959765}, interpreted in preprojective-geometric terms by Geiss and Schr\"oer \cite{MR2115084}, provided foundational examples related to non-rigidity.
Rigidity is the geometric shadow of deep representation-theoretic properties: in type $A$, 
under the Zelevinsky correspondence, it is conjecturally equivalent to $\square$-irreducibility of the associated irreducible representation of $\GL_n$ over a non-archimedean local field
(see \cite{MR4847251}*{\S3}).
Therefore it is pleasing that there is an efficient deterministic algorithm which decides rigidity at least in the multiplicity-free case.

\section{Group-level procedure}\label{sec:setup}

\subsection{Notation and conventions}
Throughout, $K$ is an algebraically closed field of
\emph{arbitrary} characteristic, and $G$ is a connected solvable linear algebraic group over $K$.
We refer to \cite{MR1102012} for standard facts about algebraic groups, including solvable ones.
In particular, by convention all algebraic groups are \emph{reduced} unless stated otherwise.

Denote by $\rad{G}$ the unipotent radical of $G$. It consists of the unipotent elements of $G$.
Denote by $G^{\diag}$ the quotient $G/\rad{G}$, which is a torus.
We fix a maximal torus $T\simeq G^{\diag}$, so that $G=T\ltimes \rad{G}$.
The pullback $\chi\mapsto\tilde\chi$ identifies $X^*(T)$ with $X^*(G)$.
(The inverse map is the restriction.)

For a closed (not necessarily connected) subgroup $H$ of $G$ we will write $\rad{H}$ for the closed subgroup
of unipotent elements, i.e., $H\cap \rad{G}$.\footnote{Milne's notation is $G_{\operatorname{u}}$ but we prefer not to use
it in order not to confuse with the stabilizer.}
In characteristic zero, or if $H$ is connected, this coincides with the unipotent radical
of $H$, i.e., the largest normal connected subgroup of $H$. However, in characteristic $p>0$ we may have
non-connected (reduced) unipotent groups, such as the kernel of the additive homomorphism $x\mapsto x^p-x$.

Because $H$ is triangularizable, the quotient $H^{\diag}=H/\rad{H}$ is diagonalizable, and 
we have a decomposition $H=S\ltimes \rad{H}$ where $S$ is a maximal diagonalizable subgroup of $H$, which is unique up to conjugation by $\rad{H}$
\cite{MR3729270}*{Ch. 16}.
We define the rank by
\[
\rk H=\dim H^{\diag}.
\]

Let $\charexp$ be the characteristic exponent of $K$.  For a sublattice $A\subset X^*(T)$ set
\[
 \Sat(A)=\{\chi:n\chi\in A\text{ for some }n\ge1\},\qquad
 \pSat(A)=\{\chi:\charexp^r\chi\in A\text{ for some }r\ge0\}.
\]
Thus $\pSat(A)=A$ in characteristic zero.  Since all algebraic groups in
this paper are reduced, kernels of characters are understood with their reduced
structure.  The assignment
\begin{equation} \label{eq: psatrec}
 S\longmapsto\{\chi\in X^*(T):\chi\rest_S=1\}
\end{equation}
is then an inclusion-reversing bijection from reduced closed subgroups of $T$ to
$\charexp$-saturated sublattices of $X^*(T)$; its inverse sends $A$ to
$\bigcap_{\chi\in A}\ker\chi$.

Let $U$ be a rational representation of $G$ of dimension $N$.
By the Lie--Kolchin theorem, there exists a flag
\[
U=F_0\supset F_1\supset\dots\supset F_N=0
\]
of $G$-submodules, which we fix.
For every $k=1,\dots,N$ let $\tilde\chi_k\in X^*(G)$ be the character by which $G$ acts on $F_{k-1}/F_k$.
Fix a basis $e_1,\dots,e_N$ of $T$-eigenvectors of $U$ with $e_k\in F_{k-1}\setminus F_k$.
Thus, $F_k=\langle e_{k+1},\dots,e_N\rangle$, and the $T$-weight of $e_k$ is $\chi_k$.

For $\lambda\in K^N$ put $v_\lambda:=\sum_j\lambda_je_j$.
We write $\supp\lambda$ for the support of $\lambda$, i.e., the set
\[
\supp\lambda=\{j:\lambda_j\ne0\}.
\]
Likewise, for $v\in U$ we write $\supp v$ for the support of the coordinates of $v$ with respect to $e_1,\dots,e_N$,
that is $\supp v_\lambda=\supp\lambda$.

Set $U_k=U/F_k$, $k=0,\dots,N$. A basis is given by the images of $e_1,\dots,e_k$ in $U_k$. The support of an element $v\in U_k$ is defined
in a similar way. It is a subset of $\{1,\dots,k\}$.

\subsection{The affine-fiber lemma}
The fibers of the projections $\pi_k\colon U_k\to U_{k-1}$ are affine lines, parameterized by the $e_k$-coordinate.
Denote by $\hat w$ the base point of the fiber with $e_k$-coordinate $0$, so that
$\pi_k^{-1}(w)$ is the image of the affine map $\alpha_w(s)=\hat w+se_k$, $s\in K$.
The following lemma is standard but crucial.

\begin{lemma}\label{lem:fiber}
Let $w\in U_{k-1}$ and $H=G_w$ its stabilizer.
Then, $H$ acts on the fiber $\pi_k^{-1}(w)$ over $w$ by affine transformations with linear part $\tilde\chi_k$, i.e.
\[
h\cdot\alpha_w(s)=\alpha_w(\tilde\chi_k(h)s+\beta_w(h)),\ \ \ h\in H, s\in K
\]
where $\beta_w(h)$ does not depend on $s$. In particular,
\begin{enumerate}
\item \label{part: action} $\beta_w$ restricts to a homomorphism $\rad{H}\rightarrow K$, so that $\rad{H}$ acts by translation on the fiber.
\item \label{part: torusaction} Let $S\subseteq T$ be a subgroup such that $\chi_j(S)=1$ for every $j\in\supp w$. Then, $S$ fixes
$w$ and $\hat w$ and acts linearly on the fiber: $t\cdot\alpha_w(s)=\alpha_w(\chi_k(t)\,s)$.
\item \label{part: Bfinite} Suppose that $\beta_w\rest_{\rad{H}}$ is neither $0$ nor surjective.
(This can only happen if $\chr K>0$.) Then, $\chi_k$ is in $\Sat(\langle\chi_j:j\in\supp w\rangle)$.
\end{enumerate}
\end{lemma}

\begin{proof}
Only the last part is non-obvious. Let $S=\cap_{j\in\supp w}\ker\chi_j\subset G_w$.
For $t\in S$ and $u\in \rad{H}$ the affine formula gives
\[
\beta_w(tut^{-1})=\chi_k(t)\beta_w(u).
\]
Let $B$ be the image of $\beta_w\rest_{\rad{H}}$. By assumption, $B$ is neither $0$ nor $\Addgrp$.
Since $B$ is a closed subgroup of $\Addgrp$, $B$ must be finite.
Since $S$ normalizes $\rad{H}$, it follows that $\chi_k(t)B=B$ for every $t\in S$.
Therefore, $\chi_k(S)$ is contained in the finite group
\[
\{c\in K^\times:cB=B\}.
\]
It follows that $\chi_k$ is trivial on $S^\circ$, and hence $\chi_k$ is in $\Sat(\langle\chi_j:j\in\supp w\rangle)$.
\end{proof}

\subsection{The algorithm}
We now describe the greedy procedure.

\begin{definition}[the greedy procedure] \label{def:greedy}
Define inductively subsets $\Phi_k\subset\{1,\dots,N\}$, $k=0,\dots,N$ as follows.
Set $\Phi_0:=\emptyset$.
For $k=1,\dots,N$ let $\Phi_k=\Phi_{k-1}$ if $G\cdot v$ is dense in $U_k$
for any (or all) $v\in U_k$ with support $\Phi_{k-1}$.
Otherwise, set $\Phi_k=\Phi_{k-1}\cup\{k\}$ (an \emph{activation});
if the characters $\chi_j$, $j\in\Phi_k$, are $\Z$-dependent \emph{abort}.

If all $N$ steps complete, the procedure \emph{succeeds}; its output is the \emph{active set}
$\Phi:=\Phi_N\subseteq\{1,\dots,N\}$, the corresponding sets $\Psi=\{\chi_j:j\in\Phi\}$,  $\tilde\Psi=\{\tilde\chi_j:j\in\Phi\}$ of \emph{active weights}
and the subgroup
\begin{equation} \label{def: Spsi}
S_\Psi=\cap_{\chi\in\Psi}\ker\chi\subset T.
\end{equation}
\end{definition}

Note that for every pending step $k$ the characters $\chi_j$, $j\in\Phi_{k-1}$, are $\Z$-independent (otherwise the run would already
have aborted), so the algorithm does not depend on the choice of $v$ with support $\Phi_{k-1}$
since $T$ acts transitively on the set of these vectors.
Likewise, it does not depend on the choice of $e_i$'s. However, it a priori depends on the choice of invariant flag.
In fact, $\Psi$ itself depends on this choice, although as we will soon see the success of the procedure does not.

For every $k$ define the subgroup
\[
S_k:=\cap_{j\in\Phi_k}\ker\chi_j\subset T.
\]

The key fact for analyzing the procedure is the following
\begin{lemma} \label{lem: kstep}
Let $k$ be a step reached by the procedure.
Let $v\in U_k$ have support $\Phi_k$, put $H=G_v$, $w=\pi_k(v)\in U_{k-1}$ and $H'=G_w$. Then,
\begin{enumerate}
\item The step $k$ is activated if and only if $\rad{H'}$ does not act transitively on the fiber
$\pi_k^{-1}(w)$. In characteristic $0$ this is equivalent to $\rad{H'}$ acting trivially on the fiber.
\item If $k$ is activated, then the following are equivalent.
\begin{enumerate}
\item The process aborts at $k$.
\item $G\cdot v$ is not dense in $U_k$.
\item $S_k^\circ=S_{k-1}^\circ$.
\end{enumerate}
\item If $\rad{H'}$ acts either trivially or transitively on the fiber (which always happens in characteristic $0$), then $H=S_k\ltimes \rad{H}$.
In particular, $S_k$ is a maximal diagonalizable subgroup of $H$.
\item If $\rad{H'}$ acts neither transitively nor trivially on the fiber (which can only happen in positive characteristic)
then $k$ activates and aborts and $(G_{v'})^\circ=(H')^\circ$ for every $v'\in\pi_k^{-1}(w)$.
\end{enumerate}
\end{lemma}

\begin{proof}
Clearly, $S_k\subset H$.

We use induction on $k$. 
By induction hypothesis (if $k>1$) $w$ has a dense orbit in $U_{k-1}$ and $H'=S_{k-1}\ltimes \rad{H'}$.
These statements hold trivially for $k=1$ as well.
Clearly $H\subset H'$. Moreover, since $G\cdot w$ is open in $U_{k-1}$,
\[
\dim H'\le\dim H+1,
\]
with equality precisely when $\dim G\cdot v=\dim U_k$, i.e. when the $G$-orbit of $v$ is dense in $U_k$.
Thus, $G\cdot v$ is dense in $U_k$ if and only if $H^\circ\ne (H')^\circ$.
Consider the action of $H'$ on the fiber $\pi_k^{-1}(w)$. 

By \cref{lem:fiber} \cref{part: action}, $\rad{H'}$ acts by translations on $\pi_k^{-1}(w)$.
The image $B$ of $\beta_w\rest_{\rad{H'}}$ in $K$ can be either finite or all of $K$.
There are three cases to consider:
\begin{enumerate}
\item $B=K$.
\item $B=0$.
\item $B$ is finite. (This can only happen if $\chr K>0$ because in characteristic $0$ all unipotent groups are connected.)
\end{enumerate} 
In the first case, $\rad{H'}$ acts transitively on $\pi_k^{-1}(w)$.
Since $G\cdot w$ is dense in $U_{k-1}$, $G\cdot v'$ is dense in $U_k$ for every $v'\in\pi_k^{-1}(w)$.
In particular this holds for $v'=\hat w$. Thus, $\Phi_k=\Phi_{k-1}$ and therefore $v=\hat w$ and $S_k=S_{k-1}$.
Since $S_{k-1}\subset H$ we obtain
\[
H=(S_{k-1}\ltimes \rad{H'})\cap H=S_{k-1}\ltimes (\rad{H'}\cap H)=S_{k-1} \ltimes \rad{H}=S_k\ltimes \rad{H}.
\]

Suppose now that $\rad{H'}$ acts trivially on $\pi_k^{-1}(w)$.
Then, $\rad{(G_{v'})}=\rad{H'}$ for all $v'\in\pi_k^{-1}(w)$.
Thus,
\[
G_{v'}=(S_{k-1}\ltimes \rad{H'})\cap G_{v'}=(S_{k-1}\cap G_{v'})\ltimes  \rad{H'}.
\]

We use \cref{lem:fiber} \cref{part: torusaction}.
First, it implies that $S_{k-1}\subset G_{\hat w}$ and hence $H'=G_{\hat w}$, so that $G\cdot\hat w$ is not dense.
Thus, $\Phi_k=\Phi_{k-1}\cup\{k\}$. 
Moreover, $S_{k-1}\cap H=S_{k-1}\cap\ker\chi_k=S_k$ since $v$ has non-zero $e_k$-coordinate.
Thus, $H=S_k\ltimes \rad{H'}=S_k\ltimes \rad{H}$.
Therefore, $G\cdot v$ is not dense in $U_k$ if and only if $S_k^\circ=S_{k-1}^\circ$, i.e., if and only if $S_{k-1}^\circ\subset\ker\chi_k$.
This happens if and only if $\chi_k$ is in $\Sat(\langle\chi_j:j\in\Phi_{k-1}\rangle)$ or in other words, if and only if the procedure aborts.

Finally, suppose that $B$ is finite but non-zero.
Since $B$ is finite, the image of $(\rad{H'})^\circ$ under $\beta_w$ is trivial.
Also, by \cref{lem:fiber} \cref{part: Bfinite}
$\chi_k$ is trivial on $S_{k-1}^\circ$.
Hence, $(H')^\circ=S_{k-1}^\circ\ltimes (\rad{H'})^\circ\subset G_{v'}$
for every $v'\in\pi_k^{-1}(w)$, so that $G_{v'}^\circ=H'^{\circ}$. Therefore $G\cdot v'$ is not dense.
Thus, $k$ activates and aborts and $S_k^\circ=S_{k-1}^\circ$.

The lemma follows.
\end{proof}

\subsection{Consequences -- correctness, generic stabilizer and minimum support}
\begin{corollary} \label{cor: algo}
If the procedure succeeds, then there exists $v\in U$ with a dense $G$-orbit $\orbit$.
In fact, every $v\in U$ with $\supp v=\Phi$ has a dense $G$-orbit and the stabilizer $G_v$ has a maximal diagonalizable subgroup $S_\Psi$.
In contrast, if the procedure aborts at $k$, then there is no dense $G$-orbit in $U$, or in fact in $U_k$.
In particular, the verdict of the procedure (success or fail) does not depend on the choice of flag or the $e_i$'s.
\end{corollary}

\begin{proof}
The first part is clear from \cref{lem: kstep}. For the converse, suppose on the contrary that the procedure aborts at $k$ but some $v\in U_k$ has a dense orbit in $U_k$.
Then, $v$ has a dense orbit in $U_{k-1}$, and therefore upon moving $v$ by $G$ we may assume that $\supp\pi_k(v)=\Phi_{k-1}$ (since the procedure did not abort before $k$).
Now, either $v=\hat w$ or $\supp v=\Phi_k$. The first possibility cannot hold because by assumption $G\cdot\hat w$ is not dense ($k$ is activated).
Therefore $\supp v=\Phi_k$ and we get a contradiction to \cref{lem: kstep}.
\end{proof}

\begin{corollary} \label{cor: independence}
Suppose that the procedure succeeds, i.e., $G$ has an open orbit $\orbit$.
Then,
\begin{enumerate}
\item For any $w\in\orbit$, the images of $G_w$ and $S_\Psi$ in $G^{\diag}$ coincide. Thus, $S_\Psi$, and therefore
$\pSat(\Psi)$, do not depend on the choice of flag and eigenbasis.
\item For every $w\in\orbit$,
\[
\#\Phi=\rk G-\rk G_w=N-\dim(\rad{G}\cdot w).
\]
In particular, the size of the active set is independent of all choices.
\item Let $w\in\orbit$ and assume that $w$ is contained in $\oplus_{\chi\in A} U_\chi$ for some subset $A$ of weights of $T$.
Then, $\Psi$ is contained in $\pSat(\langle A\rangle)$. In particular, $\# A\ge\#\Phi$.
\end{enumerate}
\end{corollary}

\begin{proof}
The first part holds since the stabilizers of two points of the orbit are conjugate and
inner conjugation acts trivially on $G^{\diag}$ since $G^{\diag}$ is abelian.
We can then recover $\pSat(\Psi)$ from $S_\Psi$ by \cref{eq: psatrec}.

Let $v\in U$ with $\supp v=\Phi$.
It is enough to check the second part in the case $w=v$ (since $\rad{G}$ is normal in $G$).

Let $H=G_v$.
Since $\Psi$ is $\Z$-independent, we have $\#\Phi=\dim T-\dim S_\Psi=\rk G-\rk H$.

For the second formula,
$N=\dim\orbit=\dim G-\dim H$, with $\dim G=\dim T+\dim \rad{G}$ and
$\dim H=\dim S_\Psi+\dim \rad{H}$ while $\dim(\rad{G}\cdot v)=\dim \rad{G}-\dim \rad{H}$.

Finally, if $w\in\oplus_{\chi\in A} U_\chi$ then $\cap_{\chi\in A}\ker\chi\subset G_w$ and hence $\langle\Psi\rangle\subset\pSat(\langle A\rangle)$ by the first part.
Since $\rk\langle\Psi\rangle=\#\Phi$ and $\rk\pSat(\langle A\rangle)=\rk\langle A\rangle\le\# A$ we get $\# A\ge\#\Phi$.
\end{proof}

\begin{remark} \label{rem: dependsonflag}
In general, in case of success the set $\Psi$ of active weights itself may depend on the choice of flag.
(See \cref{exam: flagdep} in \cref{sec: typeA}.)
It will be interesting to analyze further the structure of the possible sets of active weights as we vary the flag.
\end{remark}

\begin{remark} \label{rem: coinvactiv}
In the case $G=T$ is a torus, every step activates and we recover the standard fact that an open orbit
exists if and only if $\chi_1,\dots,\chi_N$ are $\Z$-independent, in which case the open orbit consists
of the vectors with full support.

In contrast, if $G$ is unipotent and $U\ne0$ then the first step aborts since $G$ acts trivially on $U_1$ and $X^*(G)=0$,
in accordance with the fact that there are no open orbits. In fact, as is well known, all orbits are closed in this case \cite{MR130878}*{Theorem 2}.
\end{remark}

More generally, let $U_{\coinv}$ be the space of coinvariants with respect to $\rad{G}$.
Thus $U_{\coinv}=U/U'$ where
\begin{equation} \label{eq: radcoinv} 
U'=\langle u\cdot v-v:u\in \rad{G}, v\in U\rangle.
\end{equation}
Since $\rad{G}$ is normal in $G$, $U'$ is a $G$-submodule, and $G$ acts on $U_{\coinv}$ through the torus quotient $G^{\diag}$.
The images of the $F_k$'s form a filtration of $U_{\coinv}$ whose successive quotients $(F_{k-1}+U')/(F_k+U')$ are of dimension $\le1$.
Since $F_{k-1}=\langle e_k\rangle+F_k$, the $k$-th quotient is nonzero if and only if $e_k\notin U'+F_k$, in which case it is spanned
by the image of $e_k$, on which $T$ acts by $\chi_k$.
Hence, the weights of $T$ on $U_{\coinv}$, with multiplicity, are the $\chi_k$ where $k$ ranges over the indices with $e_k\notin U'+F_k$.

\begin{lemma} \label{lem: coinvactive}
In the case of success, every index $k$ such that $e_k\notin U'+F_k$ is active.
In particular, for any choice of invariant flag and eigenbasis, the weights of $T$ on $U_{\coinv}$ are active weights.
Thus, they are $\Z$-independent, and in particular, distinct.
\end{lemma}

\begin{proof}
We prove the contrapositive: if the step $k$ is not activated, then $e_k\in U'+F_k$.
Let $w\in U_{k-1}$ with $\supp w=\Phi_{k-1}$.
By the first part of \cref{lem: kstep}, $\rad{(G_w)}$ acts transitively on the fiber $\pi_k^{-1}(w)$.
In particular, there exists $u\in\rad{(G_w)}\subset \rad{G}$ such that $u\cdot\hat w=\hat w+e_k$ in $U_k$.
Lifting $\hat w$ to a vector $v\in U$, we obtain $u\cdot v-v\in e_k+F_k$, and therefore $e_k\in U'+F_k$, as claimed.

The last assertion follows since in the case of success the characters $\chi_j$, $j\in\Phi$, are $\Z$-independent.
(Alternatively, the $\Z$-independence of the weights of $U_{\coinv}$ follows from the fact that the image of the
open $G$-orbit is a dense $T$-orbit in $U_{\coinv}$.)
\end{proof}

In view of \cref{rem: dependsonflag}, the converse of \cref{lem: coinvactive} does not hold in general: an active weight need not occur in $U_{\coinv}$.

\begin{definition} \label{def: strongPVS}
We say that $U$ is a strong {\PVS} if it admits an open orbit and $\Psi$ coincides with the set of weights of $T$ on $U_{\coinv}$. 
\end{definition}

\begin{lemma}
$U$ is a strong {\PVS} if and only if the following two conditions are satisfied.
\begin{enumerate}
\item The weights of $T$ on $U_{\coinv}$ are $\Z$-independent (and in particular, distinct).
\item For any (or equivalently, every) $w\in U_{\coinv}$ with full support,
$\rad{G}$ acts transitively on the fiber of $w$ under the projection $p:U\rightarrow U_{\coinv}$.
\end{enumerate}
\end{lemma}

\begin{proof}
By \cref{lem: coinvactive}, the weights of $T$ on $U_{\coinv}$, counted with multiplicity, occur among the active weights for every invariant flag.
By \cref{cor: independence}, the number of active weights is independent of the flag. Consequently, the condition that the active weights coincide with the coinvariant weights is flag-independent;
equivalently, it is the condition $\#\Phi=\dim U_{\coinv}$.

Let $q=\dim U_{\coinv}$.
Choose a $G$-stable complete flag with $F_q=U'$, where $U'$ is as in \cref{eq: radcoinv}.
The first $q$ flag quotients form a complete flag of $U_{\coinv}$.  Since $\rad{G}$ acts trivially on
$U/U'$, each of these steps activates; the process passes them without aborting exactly
when the corresponding torus weights are $\Z$-independent.  These are precisely
the weights of $T$ on $U_{\coinv}$.

Fix a full-support point $\bar w\in U_{\coinv}$.  For the remaining steps the
algorithm refines the affine fiber $p^{-1}(\bar w)$ by one-dimensional quotients.  If
$\rad{G}$ acts transitively on $p^{-1}(\bar w)$, then at each refinement the
stabilizer of the previously chosen quotient point acts transitively on the next affine
line: two lifts with the same previous quotient are related by an element that necessarily
stabilizes that quotient.  Hence none of the remaining steps activates.  Conversely,
stepwise transitivity along the flag lifts inductively to transitivity on the whole fiber.
Thus the active weights are exactly the coinvariant weights if and only if the two conditions of the lemma hold.
Transitivity for one full-support $\bar w$ is equivalent to transitivity
for every such $\bar w$, since the torus acts transitively on the full-support locus.
\end{proof}

\subsection{Infinitesimal/separable version}

In case $G$ has a dense orbit $\orbit$ in $U$
we may ask whether the orbit map $G\rightarrow U$ is separable, in which case we say that $G$ has an open/dense separable orbit.
Of course, this is automatic if $\chr K=0$.

Equivalently, let $\Lieg$ be the Lie algebra of $G$, acting on $U$ by the differential of the $G$-action.
The question is whether there exists $v\in U$ such that the infinitesimal orbit map
$\Lieg\rightarrow U$, $X\mapsto X\cdot v$ is onto. In this case, the set of $v$'s for which this is satisfied is the open orbit.

We can replace the greedy procedure of \cref{def:greedy} by an infinitesimal version.
(Provisionally, we denote by $\Phi'_k$ the resulting sets.)
For the non-activation, we strengthen the open orbit condition on $v\in U_k$ with $\supp v=\Phi'_{k-1}$
by requiring that $v$ has an open \emph{separable} orbit in $U_k$.
Upon activation, we strengthen the non-abortion condition to require that the images of $\chi_i$, $i\in\Phi'_k$ in $X^*(T)\otimes_{\Z}K$
are linearly independent over $K$, or equivalently, that $d\chi_i$, $i\in\Phi'_k$ are linearly independent in $\Liet^*$.

For $\chr K=0$ these conditions are identical to the original ones and the procedures coincide.
If $\chr K=p>0$ the abortion condition is that the images of $\chi_i$, $i\in\Phi'_k$ in $X^*(T)/pX^*(T)$ are linearly dependent over $\Z/p\Z$.

In both cases the activation condition is independent of the choice of $v$ since it depends only on the $G$-orbit
of $v$ and $T$ acts transitively on the $v$'s in $U_k$ with $\supp v=\Phi'_{k-1}$ since
$\chi_i$, $i\in\Phi'_{k-1}$ are $\Z$-independent in $X^*(T)$ (otherwise the procedure has already aborted).

Write the action of Lie algebra in terms of the basis $e_1,\dots,e_N$, i.e.
\[
X\cdot v=\sum_{i=1}^N\rho_i(v)(X)e_i,\ \ X\in\Lieg, v\in U.
\]
We view the coefficients as $N$ linear maps 
\[
\rho_i:U\rightarrow\Lieg^*,\ \ i=1,\dots,N.
\]

We can now write the infinitesimal version as follows.
\begin{definition}[the greedy procedure -- linear algebra version] \label{def:greedy2}
Define inductively subsets $\Phi'_k\subset\{1,\dots,N\}$, $k=0,\dots,N$ as follows.
Set $\Phi'_0:=\emptyset$. Suppose that $\Phi'_{k-1}$ has been defined. Let $v\in U_k$ have support $\Phi'_{k-1}$.
If
\[
\rho_1(v),\dots,\rho_k(v)\in\Lieg^*
\]
are linearly independent in $\Lieg^*$, set $\Phi'_k=\Phi'_{k-1}$.
Otherwise, set
\[
\Phi'_k=\Phi'_{k-1}\cup\{k\}
\]
In the latter case, abort if the characters $d\chi_i$, $i\in\Phi'_k$ are linearly dependent in $\Liet^*$
(or equivalently, the characters $d\tilde\chi_i$, $i\in\Phi'_k$ are linearly dependent in $\Lieg^*$).
\end{definition}

Note that for each $k$, $\rho_k(v)$ depends only on the image of $v$ in $U_k$, that is
\begin{subequations}
\begin{equation} \label{eq: deprhok}
\rho_j(e_k)=0\text{ for all }j<k.
\end{equation}
Moreover,
\begin{equation} \label{eq: rhok}
\rho_k(e_k)=d\tilde\chi_k.
\end{equation}
Finally, restricting to $\Liet=\Lie T$ for every $i,j=1,\dots,N$
\begin{equation} \label{eq: projrho}
\rho_i(e_j)\rest_{\Liet}=\delta_{i,j}\, d\chi_i.
\end{equation}
\end{subequations}

\begin{proposition} \label{prop: greedy}
The greedy infinitesimal procedure succeeds if and only if there exists an open separable orbit $\orbit$ in $U$.
In this case, the algorithm stepwise coincides with the non-infinitesimal version.
If the greedy infinitesimal procedure aborts at $k$ then there is no open separable orbit in $U_k$.
\end{proposition}

\begin{proof}
For any $v\in U$ let $L_k(v)$ be the span of $\rho_i(v)$, $i=1,\dots,k$ in $\Lieg^*$.
Clearly, $L_k(v)$ depends only on the image of $v$ in $U_k$.
Suppose that step $k$ is reached.
We claim that for all $j<k$ and $v\in U_j$ with $\supp v=\Phi'_j$ we have $\dim L_j(v)=j$ and
$d\tilde\chi_i\in L_j(v)$ for all $i\in\Phi'_j$. Moreover, if $k$ is activated, the abort condition is precisely that
\begin{equation} \label{eq: abortcond}
L_k(v)=L_{k-1}(v)\text{ for any $v$ with }\supp v=\Phi'_k.
\end{equation}

We prove this by induction on $k$. The case $k=0$ is empty. Assume the statement holds for $k-1$. If $k$ is not activated, the statement holds for $k$
since $L_k(v)$ properly contains $L_{k-1}(v)$ for $v\in U_k$ with $\supp v=\Phi'_k=\Phi'_{k-1}$.
Suppose that $k$ activates. Then $L_k(v)=L_{k-1}(v)$ for any $v$ with $\supp v=\Phi'_{k-1}$.
Therefore by \cref{eq: deprhok} and \cref{eq: rhok}, $L_k(v)=L_{k-1}(v)+K d\tilde\chi_k$ for any $v$ with $\supp v=\Phi'_k$.
Since $d\tilde\chi_i\in L_{k-1}(v)$ for all $i\in\Phi'_{k-1}$ by induction hypothesis, we infer (again using \cref{eq: deprhok})
that if the abort condition is satisfied then \cref{eq: abortcond} holds.
Conversely, suppose that $L_k(v)=L_{k-1}(v)$, i.e. $\rho_k(v)\in L_{k-1}(v)$, for any $v$ with $\supp v=\Phi'_k$.
The support condition on $v$ together with \cref{eq: projrho} means that the image of $L_{k-1}(v)$ under the restriction map $\pr:\Lieg^*\rightarrow\Liet^*$
is contained in the span of $d\chi_i$, $i\in\Phi'_{k-1}$.
On the other hand,
\[
\rho_k(v)\rest_{\Liet}=\lambda_k\,d\chi_k
\]
where $\lambda_k\ne0$ is the $k$-th coordinate of $v$. Therefore, the abort condition is satisfied.
Our claim follows.

We infer that if the infinitesimal procedure succeeds, then $\rho_1(v),\dots,\rho_N(v)$ are linearly independent in $\Lieg^*$
for any $v\in U$ with $\supp v=\Phi'_N$. This means that the orbit of $v$ is separably open.
On the other hand, if the infinitesimal procedure aborts at $k$, then $\dim L_k(v)=\dim L_{k-1}(v)=k-1$ for any $v\in U_k$
with $\supp v=\Phi'_k$. Suppose on the contrary that $v\in U_k$ has an open separable orbit.
Then $\pi_k(v)\in U_{k-1}$ is in the open (separable) orbit of $U_{k-1}$ and by using the $G$-action we may assume that $\supp\pi_k(v)=\Phi'_{k-1}$
so that $\Phi'_{k-1}\subset\supp v\subset\Phi'_k$. Since $k$ activates, $\supp v=\Phi'_{k-1}$ is not possible. Therefore $\supp v=\Phi'_k$.
However, this contradicts that $\dim L_k(v)<k$.

Finally, suppose that the infinitesimal procedure succeeds, i.e. an open separable orbit exists.
Then an open separable orbit exists in $U_k$ for all $k$.
We prove by induction on $k$ that $\Phi_k=\Phi'_k$.
For the induction step, one notes that a non-activation step in the original procedure is also a non-activation step
in the infinitesimal version, otherwise the resulting open orbit in $U_k$ would not be separable.
Conversely, a non-activation step in the infinitesimal procedure is also a non-activation step
in the group-level version simply because separable open implies open.

The proposition follows.
\end{proof}

\section{Application: rigid components of Lusztig's nilpotent varieties of Dynkin type}\label{sec:quiver}

As before, $K$ is an algebraically closed field.
The following standard result follows from Krull--Schmidt, Fitting's lemma and the Wedderburn--Artin theorem.

\begin{lemma} \label{lem: indecomprank}
Let $M$ be a finite-dimensional module over a $K$-algebra $A$.
Let $m_1,\dots,m_k$ be the multiplicities in the indecomposable decomposition of $M$.
Then,
\begin{enumerate}
\item The unipotent radical of $\Aut_A(M)$ is $1+\mathfrak{r}$ where $\mathfrak{r}$ is the radical of $\End_A(M)$.
\item The semisimple quotient of $\End_A(M)$ is isomorphic to $\oplus_{i=1}^k\Mat_{m_i}$.
\item The reductive quotient of $\Aut_A(M)$ is isomorphic to $\prod_{i=1}^k\GL_{m_i}$.
\end{enumerate}
In particular, the rank of $\Aut_A(M)$ is $m_1+\dots+m_k$ and $\Aut_A(M)$ is solvable if and only if $M$ is multiplicity free
i.e., $m_1=\dots=m_k=1$. In this case, writing $M=M_1\oplus\dots\oplus M_k$, the group of automorphisms that act as nonzero scalars on each $M_i$ is a maximal torus.
If moreover every $M_i$ is a brick (i.e., $\End_A(M_i)=K$) then $\mathfrak{r}=\oplus_{i\ne j}\Hom_A(M_i,M_j)$.
\end{lemma}

We will be particularly interested in the case where $A$ is the path algebra of a Dynkin quiver $Q$.
In this case the isomorphism classes of finite-dimensional $A$-modules (i.e. representations of $Q$) are parameterized by multisets of positive roots of the corresponding root system.
For any such $M$ there is a particularly important representation of $\Aut_Q(M)$. Existence of an open orbit for it is equivalent to rigidity of the irreducible component
of Lusztig's nilpotent variety pertaining to $M$.
In the case where $M$ is multiplicity free (corresponding to sets of positive roots, rather than multisets),
we can apply the discussion of \cref{sec:setup}.

\subsection{The nilpotent variety and its components}

We recall standard facts about Dynkin quivers and their representations. 
For a recent account with closely related notation, see \cite{MR4847234} and the references therein.

Let $Q=(I,\Omega)$ be a finite quiver with vertex set $I$ and arrow set
$\Omega$; an arrow $h$ runs from $s(h)$ to $t(h)$.  We assume
throughout that the underlying graph of $Q$ is a connected, simply laced
Dynkin diagram (type $A$, $D$ or $E$).  Representations of $Q$ on
an $I$-graded vector space $V=\bigoplus_{i\in I}V_i$ are identified
with modules over the path algebra $KQ$; we write $\Hom_Q$,
$\End_Q$, $\Aut_Q$, $\Ext^1_Q$ for the module notions (higher
$\Ext$'s vanish, the path algebra being hereditary) and $\Lin$ for
spaces of plain $K$-linear maps.  Let
\[
R(V):=\bigoplus_{h\in\Omega}\Lin(V_{s(h)},V_{t(h)}),\qquad
\bar R(V):=\bigoplus_{h\in\Omega}\Lin(V_{t(h)},V_{s(h)}),
\]
on which $G_V:=\prod_{i\in I}\mathrm{GL}(V_i)$ acts by conjugation, and define the moment map
\[
\mu\colon R(V)\times\bar R(V)\to\bigoplus_{i\in I}\Lin(V_i,V_i),
\qquad
\mu(x,y)_i:=\sum_{h:\,t(h)=i}x_hy_h-\sum_{h:\,s(h)=i}y_hx_h .
\]
\emph{Lusztig's nilpotent variety} in the case at hand is $\Lambda(V):=\mu^{-1}(0)\subseteq R(V)\times\bar R(V)$ \cite{MR1088333}.
Its points are the modules over the preprojective algebra $\Pi=\Pi(Q)$ with underlying graded space $V$. 
(Note that since $Q$ is Dynkin, $\Pi$ is finite-dimensional; equivalently, its arrow ideal is nilpotent.
Hence every finite-dimensional $\Pi$-module automatically satisfies Lusztig's nilpotency condition.)
Let $\pi\colon\Lambda(V)\to R(V)$ be the first projection.

By Gabriel's theorem \cite{MR0332887}, the assignment $M\mapsto\dimv M$ is a
bijection between the isomorphism classes of indecomposable
$KQ$-modules and the positive roots of the root system of the diagram.
In particular there are finitely many indecomposables and,
therefore finitely many isomorphism classes of modules of each graded dimension, so $R(V)$ is a finite union of $G_V$-orbits $O_M$.

Moreover, the indecomposable modules are bricks and rigid (i.e., $\Ext^1_Q(X,X)=0$).

For indecomposable $Q$-modules $M_1$ and $M_2$ we write $M_1\homr M_2$ if $\Hom_Q(M_2,M_1)\ne0$ and $M_1\extr M_2$ if $\Ext^1(M_1,M_2)\ne0$.
We denote the transitive closure of $\homr$ by $\homrt$. This is a partial order on the set of isomorphism classes of indecomposable $Q$-modules (i.e., positive roots).
We have
\begin{equation} \label{eq: extimplieshom}
M_1\extr M_2\implies M_1\homrt M_2.
\end{equation}
Indeed, suppose that $M_1\extr M_2$ and choose a non-split extension
\[
0\rightarrow M_2\xrightarrow{\ \iota\ }E\xrightarrow{\ \pi\ }M_1\rightarrow0.
\]
Write $E=\oplus_tE_t$ with $E_t$ indecomposable.
We claim that $\Hom_Q(M_2,E_t)\ne0$ and $\Hom_Q(E_t,M_1)\ne0$ for at least one $t$.
Otherwise, let $E'$ be the sum of those $E_t$ for which $\pi\rest_{E_t}\ne0$, and let $E''$ be the sum of the remaining ones,
so that $E=E'\oplus E''$.
By assumption, $\iota(M_2)$ has zero component in each summand of $E'$, whence $\iota(M_2)\subseteq E''$;
on the other hand $\pi(E'')=0$ and therefore $E''\subseteq\ker\pi=\iota(M_2)$.
Thus $E''=\iota(M_2)$ and $\pi$ restricts to an isomorphism $E'\rightarrow M_1$, contradicting the non-splitness of the extension.
Fixing such a $t$ we obtain $M_1\homr E_t\homr M_2$, and hence $M_1\homrt M_2$.

In particular, the relation $\extr$ is acyclic.

Let $x\in R(V)$ and let $M$ be the corresponding module.
The stabilizer of $x$ in $G_V$ is $\Aut_Q(M)$.
We identify the fiber $E:=\pi^{-1}(x)$ with a linear subspace of $\bar R(V)$, on which $\Aut_Q(M)$ acts linearly.

Recall the following standard fact.

\begin{proposition}
\begin{enumerate}
\item The elements of $E$ correspond to the $\Pi$-module structures on $M$ that are compatible as a $Q$-module.
\item If $y\in E$ corresponds to $\tilde M$, then the stabilizer in $\Aut_Q(M)$ of $y$ is $\Aut_\Pi(\tilde M)$.
\item As an $\Aut_Q(M)$-module, $E$ is isomorphic to $\Ext^1_Q(M,M)^*$.
\item Suppose that $M=\bigoplus_iM_i$ with $x$ block-diagonal. Then $E$ is
the direct sum of its blocks $E_{i,j}$ consisting of the
elements of $E$ that map $M_j$ to $M_i$ and vanish on the others.
Under the isomorphism $E\simeq\Ext^1_Q(M,M)^*$ we have $E_{i,j}\cong\Ext^1_Q(M_i,M_j)^*$.
\end{enumerate}
\end{proposition}

\begin{proof}
The first two assertions follow directly from the definition of the
preprojective algebra and of the stabilizer.  For the third, let
$\Lieg_V=\bigoplus_{i\in I}\End(V_i)$.  Since $KQ$ is hereditary, the standard
deformation complex gives an exact sequence
\[
 0\longrightarrow\End_Q(M)\longrightarrow\Lieg_V
 \xrightarrow{d_x}R(V)\longrightarrow\Ext^1_Q(M,M)\longrightarrow0,
\]
where $(d_xa)_h=a_{t(h)}x_h-x_ha_{s(h)}$.  Pair $R(V)$ with $\bar R(V)$ by
$\langle z,y\rangle=\sum_{h\in\Omega}\tr(z_hy_h)$.  A direct trace
calculation gives
\[
 \langle d_xa,y\rangle=\sum_{i\in I}\tr(a_i\mu(x,y)_i).
\]
Hence $\mu(x,y)=0$ if and only if $y$ annihilates $\operatorname{im}d_x$.  Therefore
\[
 E=\pi^{-1}(x)=(\operatorname{im}d_x)^\perp
   \simeq\operatorname{coker}(d_x)^*
   \simeq\Ext^1_Q(M,M)^*.
\]
All maps are $\Aut_Q(M)$-equivariant.  If $M=\bigoplus_iM_i$, the trace pairing pairs
opposite matrix blocks, and the block of $E$ mapping $M_j$ to $M_i$ is dual to
$\Ext^1_Q(M_i,M_j)$, proving the final assertion.
\end{proof}

For any $M$, the preimage $\pi^{-1}(O_M)$ of its orbit is the conormal bundle to $O_M\subset R(V)$, which is of dimension $\dim R(V)$, and its closure $Z_M$
is an irreducible component of $\Lambda(V)$. The map
\[
M\mapsto Z_M
\]
induces a bijection between isomorphism classes of modules of graded dimension $\dimv V$
(i.e., $G_V$-orbits in $R(V)$) and the irreducible components of $\Lambda(V)$ (\cite{MR0390138}, \cite{MR1088333}*{\S14}).
(For a general symmetric Kac--Moody datum the set of irreducible components is no longer described by
orbits, and is identified with the crystal $B(\infty)$ in \cite{MR1458969}.)
The components are thus parameterized by the multisets of positive roots summing to $\dimv V$, through
$M=\bigoplus_{\beta}M_\beta^{\oplus m_\beta}$.

If $\tilde M$ is a $\Pi$-module, identified with a point in $\Lambda(V)$, then by \cite{MR1834739}
\[
\operatorname{codim}_{Z_M}(G_V\cdot\tilde M)=\frac12\dim\Ext_\Pi^1(\tilde M,\tilde M).
\]

By definition, an irreducible component is \emph{rigid} if it contains a rigid $\Pi$-module, or equivalently,
an open $G_V$-orbit. Since
\[
\pi^{-1}(O_M)=G_V\times^{\Aut_Q(M)}E,
\]
$Z_M$ is rigid if and only if $\Aut_Q(M)$ has a dense orbit on $E=\pi^{-1}(x)$ for $x\in O_M$.

The condition is a genuine restriction: by \cite{MR1944812} a direct sum $Z_{M_1}\oplus\dots\oplus Z_{M_t}$
of components is again a component precisely when the generic $\Ext^1_\Pi$ between the summands vanishes
in both directions, and this canonical decomposition is what makes the multiplicativity statements for the
dual semicanonical basis in \cite{MR2144987} work; \cite{MR2144987}*{Lemma 7.2} is the statement that a
component with a dense orbit has its dual semicanonical vector equal to the delta function of that orbit.
\Cref{part: quiverdecomp} of \cref{thm:quiver} below computes the canonical decomposition of a rigid
$Z_M$ explicitly, from the connected components of the active forest.

Note that in the case at hand,
\begin{equation} \label{eq: autsep}
\text{the orbit map is automatically separable.}
\end{equation}
Indeed, the stabilizer group scheme is the unit group
of the finite-dimensional algebra $\End_\Pi(\tilde M)$, hence an open smooth subscheme of its
underlying affine space.

\subsection{Multiplicity-free modules and an adapted stable flag} \label{sec: convij}

From now on suppose that $M$ is a multiplicity-free $Q$-representation with decomposition $M=\oplus_{i=1}^n M_i$ into (non-isomorphic) indecomposables.
For simplicity we write $i\homr j$ to mean $M_i\homr M_j$; similarly for $i\homrt j$ and $i\extr j$.

Recall that $\Aut_Q(M)$ is solvable and a maximal torus is given by
the multiplicative group of $\oplus_{i=1}^n\End_Q(M_i)\simeq K^n$.
We denote by $\varepsilon_i$ the character of $T$ given by the $i$-th coordinate $i=1,\dots,n$.

Write
\[
E=\Ext_Q^1(M,M)^*=\oplus_{i\neq j}E_{i,j}
\]
with $E_{i,j}\cong\Ext^1_Q(M_i,M_j)^*$ which is nonzero precisely when $i\extr j$.
Recall that the diagonal blocks vanish since $M_i$ is $Q$-rigid
and that $E_{i,j}$ and $E_{j,i}$ cannot be both nonzero.
Note that $T$ acts on $E_{i,j}$ by the character $\varepsilon_i-\varepsilon_j$.
Let $d_{i,j}=\dim E_{i,j}$.

Denote the action of $\Aut_Q(M)$ on $E$ by $g\cdot y$.

Define the \emph{containment order} on the set of pairs $(i,j)$ with $i\extr j$
by: $(i,j)\sqsubseteq(k,l)$ iff $k\homrt i$ and $j\homrt l$.

Let $y\in E_{i,j}$ and $g=(g_{a,b})\in\Aut_Q(M)$.  The $(k,l)$ block of
$g\cdot y=g y g^{-1}$ is
\[
 (g\cdot y)_{k,l}=g_{k,i}\,y_{i,j}\,(g^{-1})_{j,l}.
\]
If this block is nonzero, then $\Hom_Q(M_i,M_k)\ne0$ and
$\Hom_Q(M_l,M_j)\ne0$.  Thus $k\homrt i$ and $j\homrt l$, i.e.
$(i,j)\sqsubseteq(k,l)$.  Consequently
\[
 g\cdot E_{i,j}\subseteq \bigoplus_{(k,l)\sqsupseteq(i,j)}E_{k,l}.
\]
It follows immediately that any ordering of a block basis extending $\sqsubseteq$ defines
an $\Aut_Q(M)$-stable complete flag.

Choose a basis of $E$ consisting of vectors lying in
single blocks, enumerated $w_1,\dots,w_N$ ($N:=\dim E$) so that the
blocks appear in an order extending $\sqsubseteq$ (the vectors of
each block consecutively, in any order); call such enumerations
\emph{adapted}.

Modulo the sum of the strictly larger blocks, $g$ acts on $E_{i,j}$ by the scalar $g_{i,i}g_{j,j}^{-1}$; hence any complete flag within a block is stable.

Then the subspaces
$F_k:=\langle w_{k+1},\dots,w_N\rangle$ form a complete
$\Aut_Q(M)$-stable flag of $E$, and $w_1,\dots,w_N$ is a compatible
eigenbasis of $T$, the character of the $k$-th flag line being
$\varepsilon_{i_k}-\varepsilon_{j_k}$, where $(i_k,j_k)$ is the
block of $w_k$.

The dependence of characters can be read off combinatorially from the following standard fact.

\begin{lemma}\label{lem:cycle}
Let $\grph$ be an undirected graph with vertices $\{1,\dots,n\}$ and edges $\{a_i,b_i\}$, $i=1,\dots,m$.
Let $T$ be the $n$-dimensional standard torus $T=\Multgrp^n$.
Then $\grph$ is acyclic if and only if 
the characters $\varepsilon_{a_i}-\varepsilon_{b_i}$ in $X^*(T)\otimes K$,  are linearly independent over $K$.
In particular, this condition is independent of the characteristic of $K$.

In general, the characters $\varepsilon_{a_i}-\varepsilon_{b_i}$ span a $\Z$-saturated lattice, and their joint kernel
is the subtorus consisting of the elements whose coordinates are constant on each connected component.
\end{lemma}

We do not need to allow multigraphs here: a repeated edge is a dependence, and is therefore ruled out by
the abort condition before it can occur.

Indeed, a cycle gives a signed linear dependence, while in a forest, we can eliminate a leaf and use induction.

\subsection{The criterion for rigid components}

\begin{theorem}\label{thm:quiver}
Let $Q$ be a Dynkin quiver and $M$ a multiplicity-free module with
root set $\Sigma$, $\#\Sigma=n$, and let $Z_M$ be the corresponding
irreducible component.  Fix an adapted enumeration of a block basis
of $E$ and run the greedy procedure of \cref{def:greedy} or its linear algebraic version
\cref{def:greedy2} for the action of $\Aut_Q(M)$ on $E$.  Then,
\begin{enumerate}
\item \label{part: quiveract} At step $k$, activation occurs if and only if the composition of
\[
\End_Q(M)\to E/F_k,\ \ \  \varphi\mapsto [\varphi,y]\mod F_k
\]
is not surjective, at some (or equivalently, every) $y$  supported exactly on $\Phi_{k-1}$.
\item \label{part: quiverabort} If $k$ is activated, the abort rule at step $k$ takes the combinatorial form: 
abort if and only if $\{i_k,j_k\}$ is already an edge or completes a cycle
in the (acyclic) graph with vertices $\Sigma$ and edges $\{i_t,j_t\}$, $t\in\Phi_{k-1}$.
\item \label{part: quiversucc} The procedure succeeds if and only if $Z_M$ is rigid.
\item \label{part: quiverforest} On success the active edges form a forest on $\Sigma$,
so $\#\Phi\le n-1$; every $y\in E$ supported exactly on $\Phi$ lies in the open
$\Aut_Q(M)$-orbit of $E$, and the $G_V$-orbit of $(x,y)$ is then
open in $Z_M$.
\item \label{part: quiverrank} $\#\Phi=n-r$, where $r$ is the rank of the stabilizer $\Aut_\Pi(\tilde M)$ in $\Aut_Q(M)$ of a point of the open orbit,
viewed as a $\Pi$-module $\tilde M$, and the partition of $\Sigma$ into the connected
components of the active forest depends on no choice made.
\item \label{part: quiverdecomp} Let $\Sigma_1,\dots,\Sigma_r$ be the vertex sets of the connected components of active forest on $\Sigma$.
Let $M=\oplus_{j=1}^r L_j$ be the corresponding decomposition where $L_j=\oplus_{i\in\Sigma_j}M_i$.
Fix $y\in E$ with support $\Phi$ and let $\tilde M$ be the corresponding $\Pi$-module.
Then every $L_j$ is a $\Pi$-submodule of $\tilde M$ and
$\tilde M=\oplus_{j=1}^r L_j$ is the decomposition of $\tilde M$ into $\Pi$-indecomposables.
\end{enumerate}
\end{theorem}

\begin{proof}
The differential at $y$ of the $\Aut_Q(M)$-orbit map is
$\varphi\mapsto[\varphi,y]$, so \cref{part: quiveract} is the infinitesimal activation criterion of
\cref{prop: greedy}.  The orbit maps are separable by \cref{eq: autsep}, hence the infinitesimal and
group-level procedures coincide.  \Cref{part: quiverabort} follows from \cref{lem:cycle}: the active characters
are the differences $\varepsilon_{i_t}-\varepsilon_{j_t}$, and adjoining the new
character preserves linear independence exactly when the new edge neither repeats an
edge nor creates a cycle.  \Cref{prop: greedy} together with
$\pi^{-1}(O_M)\simeq G_V\times^{\Aut_Q(M)}E$ proves \cref{part: quiversucc}, and \cref{cor: algo} gives
\cref{part: quiverforest}.

For \cref{part: quiverrank}, \cref{cor: independence} identifies the image of a generic stabilizer in the torus
quotient with the joint kernel $S_\Psi$ of the active characters.  By \cref{lem:cycle} this is
the subtorus whose coordinates are constant on each connected component of the active
forest.  Hence the component partition, and its number $r=\rk
\Aut_\Pi(\widetilde M)$, is independent of the adapted flag.  A forest on $n$ vertices
with $r$ components has $n-r$ edges, giving $\#\Phi=n-r$.

Finally, if $\Sigma_1,\dots,\Sigma_r$ are the components, the support condition on $y$
shows that $L_j=\bigoplus_{i\in\Sigma_j}M_i$ is stable under all reverse-arrow maps, hence
is a $\Pi$-submodule.  Thus $\widetilde M=\bigoplus_jL_j$.  \Cref{lem: indecomprank} says that the total
number of indecomposable summands of $\widetilde M$, counted with multiplicity, equals
the rank $r$ of its automorphism group.  Since the displayed decomposition already has
$r$ nonzero summands, every $L_j$ is indecomposable.
\end{proof}

We will say that $Z_M$ is strongly rigid if $E$ is a strong {\PVS} with respect to $\Aut_Q(M)$.

\section{Quivers of type \texorpdfstring{$A$}{A}} \label{sec: typeA}

Let us now specialize to the case where $Q$ is the equioriented quiver of type $A_m$:
\[
1\to2\to\dots\to m
\]
The indecomposable modules are the
\emph{segment modules} $M_{[a,B]}$, $1\le a\le B\le m$:
\begin{gather*}
(M_{[a,B]})_v=\begin{cases}Ku_v&\text{ for }v\in[a,B],\\0&\text{otherwise}\end{cases}\\
\text{the arrow $v\to v+1$ acting by $u_v\mapsto u_{v+1}$ (and $u_B\mapsto0$).}
\end{gather*}

In this case
\[
M_{[a,B]}\homr M_{[a',B']}\text{ (i.e., $\Hom_Q(M_{[a',B']},M_{[a,B]})\neq0$) if and only if }a\le a'\le B\le B',
\]
in which case  $\Hom_Q(M_{[a',B']},M_{[a,B]})$ is spanned by the map $u_v\mapsto u_v$ ($a'\le v\le B$), $u_v\mapsto0$ ($v>B$).

Similarly,
\[
M_{[a,B]}\extr M_{[a',B']}\text{ (i.e., $\Ext_Q^1(M_{[a,B]},M_{[a',B']})\neq0$) if and only if }a+1\le a'\le B+1\le B',
\]
in which case  $\Ext_Q^1(M_{[a,B]},M_{[a',B']})$ is one-dimensional.

\subsection{Incidence matrices}

Recall that an \emph{incidence matrix} is a zero-one matrix with no zero rows or columns.
We do not specify the size of the matrix, but it is of course not bigger than $n\times n$.
Every permutation matrix is an incidence matrix.
The number $I_n$ of incidence matrices with $n$ ones is the sequence \href{https://oeis.org/A101370}{OEIS A101370}.
In \cite{MR2255427} it was shown that
\[
I_n\sim cn!\beta^n\ \ \text{where }c=\tfrac14\beta e^{-\frac1{2\beta}}\text{ and }\beta=\tfrac1{(\log 2)^2}>1.
\]

Let $\IM$ be an incidence matrix with $n$ ones.
List the positions of the ones in $\IM$ as $(a_i,b_i)$, $i=1,\dots,n$, in lexicographic order.
Take $m=2n$ and set
\[
B_i:=b_i+n,\qquad S_i:=M_{[a_i,B_i]},\qquad
M_{\IM}:=S_1\oplus\dots\oplus S_n
\]
i.e., the module pertaining to the multisegment $[a_i,b_i+n]$, $i=1,\dots,n$.
Clearly $M_{\IM}$ is multiplicity-free.  Note that $a_i\le n<n+1\le B_j$ for all $i,j$.
Therefore, any two of the segments overlap.

We will say that $\IM$ is rigid/strongly rigid if $Z_{M_{\IM}}$ is rigid/strongly rigid. 

Conforming to the conventions of \cref{sec: convij}, we shall write $i\homr j$ (resp., $i\extr j$) if $\Hom_Q(S_j,S_i)\ne0$ (resp., $\Ext^1_Q(S_i,S_j)\ne0$).
Explicitly
\[
i\homr j\iff a_i\le a_j\text{ and }b_i\le b_j,\ \ \ \ i\extr j\iff a_i<a_j\text{ and }b_i<b_j.
\]
In particular, $\homr$ is already transitive, so we do not need to distinguish between $\homr$ and $\homrt$.
Note that $\extr$ is smaller than the strict part of $\homrt$. Also, because we took lexicographic order on the ones,
$i\homr j$ implies $i\le j$.

Let $\V=\{(i,j):i\homr j\}$ and $\U=\{(i,j):i\extr j\}$.

We can identify $\End_Q(M_{\IM})$ with the incidence algebra of $\homr$, namely the subalgebra of $\mathrm{Mat}_n(K)$ given by the span
$\Lieg:=\langle e_{ij}:(i,j)\in\V\rangle$.
The span
$\uu:=\langle e_{ij}:(i,j)\in\U\rangle$ is a two-sided ideal of $\Lieg$.
Thus, we can identify $G=\Aut_Q(M_{\IM})$ with the set of elements of $\Lieg$ with
invertible diagonal part.
The unipotent radical of $G$ is $\rad{G}=1+\nn$ where
$\nn:=\langle e_{ij}:(i,j)\in\V,\ i\neq j\rangle$, and a maximal torus
$T$ is given by the invertible diagonal matrices.  $G$ acts on $\uu$ by
conjugation. Moreover $E\cong\uu$ compatibly with the actions:
under the trace pairing the dual of the extension block for $i\prec j$ is the matrix unit
$e_{ij}$, and the $\Aut_Q(M_\IM)$ action becomes conjugation.

Consequently $\IM$ is rigid if and only if $G$ has a dense orbit on $\uu$.

\begin{remark}
The restriction to incidence matrices entails no loss for rigidity among
multiplicity-free modules whose segments are pairwise overlapping.
Replacing the distinct beginning and end points by their ranks preserves
the relations $\vdash$ and $\prec$, and hence the incidence-algebra pair
$(\Lieg,\uu)$ described above. With the canonical generators
of the one-dimensional Hom and Ext spaces, it therefore preserves the
conjugation representation. Thus the original module is rigid if and
only if the corresponding incidence matrix is rigid.
\end{remark}

There are two symmetries which preserve rigidity and strong rigidity.

\begin{lemma} \label{lem: symmetry}
The rigidity/strong rigidity of $\IM$ is preserved under taking transpose and under rotation by $180^\circ$ (i.e. upending the rows and the columns).
\end{lemma}

\begin{proof}
Transposition interchanges the two coordinate orders on the ones of
$\IM$; it therefore relabels both the incidence algebra $\Lieg$ and its ideal
$\uu$, giving an isomorphism of conjugation representations.  Rotation by
$180^\circ$ reverses both coordinate orders and gives an algebra anti-isomorphism
$\theta:\Lieg\to\Lieg'$, carrying $\uu$ to $\uu'$.  The maps
\[
 g\longmapsto\theta(g^{-1}),\qquad x\longmapsto\theta(x)
\]
intertwine the two conjugation actions because
$\theta(gxg^{-1})=\theta(g^{-1})\theta(x)\theta(g)$.  Thus the representations are
isomorphic, up to passage to the opposite algebra, and both rigidity and strong rigidity
are preserved.
\end{proof}

Let us make the procedure of \cref{thm:quiver} explicit for $M_{\IM}$ in the case at hand.

For $\lambda\in K^{\U}$ put $L(\lambda):=\sum_{x\in\U}\lambda_x\,e_x\in\uu$.

The matrix $M(\lambda)$ of the map $\Lieg\to\uu$, $X\mapsto[X,L(\lambda)]$, in the
bases $(e_v)_{v\in\V}$ of $\Lieg$ and $(e_u)_{u\in\U}$ of $\uu$,
has $\bigl((i,j),(i,j')\bigr)$-entry $\lambda_{(j',j)}$ if $(j',j)\in\U$;
$\bigl((i,j),(i',j)\bigr)$-entry $-\lambda_{(i,i')}$ if
$(i,i')\in\U$; and all other entries $0$.

$\rk M(\lambda)=\dim G\cdot L(\lambda)$ (conjugation orbit), and $\IM$
is rigid if and only if $G$ has a dense orbit on $\uu$.

Define the \emph{containment order} on $\U$:
$y\sqsubseteq x$ iff $i_x\homr i_y$ and $j_y\homr j_x$, where
$x=(i_x,j_x)$, $y=(i_y,j_y)$ --- the endpoints of $x$ weakly enclose
those of $y$.  Call an enumeration $x_1,\dots,x_N$ of $\U$
($N:=\#\U$) \emph{adapted} if it is a linear extension of
$\sqsubseteq$.

It is easy to see the following
\begin{proposition}\label{prop:adapted}
\begin{enumerate}
\item The parameters occurring in the row of $M(\lambda)$
indexed by $x$ are $\lambda_x$ itself (in the two diagonal columns,
with coefficients $\pm1$) and parameters $\lambda_y$ with
$y\sqsubset x$ properly.  Conversely, the transitive closure of the
relation ``$\lambda_y$ occurs in the row of $x$'' is the full
containment order.  In particular the enumerations under which every
parameter in the row of $x_i$ has index $\le i$ are exactly the
adapted ones.
\item For $g\in G$ and $y\in\U$,
$g\,e_y\,g^{-1}\in K^\times e_y+
\langle e_z:z\sqsupset y\ \text{properly}\rangle$.  Consequently,
for every adapted enumeration the subspaces
$F_k:=\langle e_{x_j}:j>k\rangle$ form a complete $G$-stable flag of
$\uu$, and the $e_x$ are $T$-eigenvectors compatible with it, with
characters
\[
\chi_{(i,j)}\;=\;\varepsilon_i-\varepsilon_j
\qquad(\text{where }\varepsilon_a(t)=t_a\ \text{for }t\in T).
\]
\end{enumerate}
\end{proposition}

\begin{proof}
\begin{enumerate}
\item By the entry description above, the
row of $x=(i,j)$ contains $\lambda_{(j',j)}$ for $i\homr j'$
(type~I; the case $j'=i$ is $\lambda_x$ in the diagonal column
$(i,i)$) and $\lambda_{(i,i')}$ for $i'\homr j$ (type~II; the case
$i'=j$ is $-\lambda_x$ in the diagonal column $(j,j)$); the others
are proper containments sharing an endpoint.  Conversely a proper
containment $y\sqsubset x$ factors through $w:=(i_y,j_x)$, which
lies in $\U$, and each of
$y\sqsubseteq w\sqsubseteq x$ is of one of the two displayed
shared-endpoint types.

\item The correction terms of $g\,e_y\,g^{-1}$ are spanned by
products $e_v\,e_y\,e_w$ with $v=(a,i_y)$, $w=(j_y,c)\in\V$, which
land at $(a,c)\sqsupseteq y$, properly unless $v,w$ are both
diagonal; the diagonal contribution is
$(t_{i_y}/t_{j_y})\,e_y\in K^\times e_y$.  For an adapted
enumeration, if $y\in\{x_{k+1},\dots,x_N\}$ then every proper
$z\sqsupset y$ comes later still, so $g\,e_y\,g^{-1}\in F_k$: each
$F_k$ is $G$-stable.  The diagonal torus acts on $e_{(i,j)}$ by
$t\,e_{(i,j)}\,t^{-1}=(t_i/t_j)\,e_{(i,j)}$, giving the weights,
and the character of $G$ on $F_{k-1}/F_k$ restricts to that weight.
\end{enumerate}
\end{proof}

Combining with \cref{thm:quiver} we get the following explicit algorithm.

\begin{theorem}\label{thm:rigidity}
Fix an enumeration $x_1,\dots,x_N$ of $\U$ adapted to containment
and run the greedy procedure: $\Phi_0:=\emptyset$; at step $i$, if
the rows $x_1,\dots,x_i$ of $M(\lambda)$ are linearly dependent for
one (equivalently, every) $\lambda$ with
$\supp\lambda=\Phi_{i-1}$, set $\Phi_i:=\Phi_{i-1}\cup\{x_i\}$, and
otherwise set $\Phi_i:=\Phi_{i-1}$; abort if $\Phi_i$, viewed as a
set of edges on the $n$ ones, contains a cycle.  If all $N$ steps
complete, the procedure \emph{succeeds}, with \emph{active set}
$\Phi:=\Phi_N$.  Then:
\begin{enumerate}
\item \label{part: rigsucc} The procedure succeeds if and only if $\IM$ is rigid.
\item \label{part: rigforest} On success, $\Phi$ is a forest
with at most $n-1$ edges, and $\rk M(\lambda)=\#\U$ for
\emph{every} $\lambda$ with $\supp\lambda=\Phi$ --- in particular
for the indicator vector of $\Phi$;
\item \label{part: rigabort} On abort at step $i$, the rows $x_1,\dots,x_i$ of
$M(\lambda)$ are linearly dependent for \emph{every} $\lambda$; in
particular $\rk M(\lambda)<\#\U$ identically.
\end{enumerate}
\end{theorem}

As noted before, in case of success, the algorithm also gives the indecomposable decomposition with respect to $\Pi$.
Write $P(\Phi)$ for the partition of the $n$ ones into connected components of the active forest $\Phi$.

\begin{theorem}\label{thm:rigidityinv}
Suppose that $\IM$ is rigid.  Then, for every adapted enumeration:
\begin{enumerate}
\item \label{part: invcount} $\#\Phi=\#\U-\varrho$, where $\varrho$ is the generic rank
of $M(\lambda)$ with all diagonal columns deleted (equivalently
$\varrho=\dim[\nn,L(\lambda)]$ for generic $\lambda$);
\item \label{part: invpartition} The partition $P(\Phi)$ does not depend on the enumeration:
it is the partition of the ones by the subtorus $S_\Psi\subseteq T$ of \cref{def: Spsi}
 --- diagonal tori constant on the blocks --- i.e.\ $a\sim b$ iff $t_a=t_b$ for all $t\in S_\Psi$;
\item \label{part: invblocks} The number of blocks of $P(\Phi)$ equals the rank $r$ of
the stabilizer of the dense orbit, and $\#\Phi=n-r$.
\end{enumerate}
\end{theorem}

\begin{proof}
Let $L=L(\lambda)$ be a point in the open orbit.  \Cref{cor: independence} gives
\[
 \#\Phi=\dim\uu-\dim((1+\nn)\cdot L)=\#\U-\dim[\nn,L]=\#\U-\varrho,
\]
which proves \cref{part: invcount}.  The same corollary identifies the image of the generic stabilizer in
$T$ with $S_\Psi$.  Since the active characters are
$\varepsilon_i-\varepsilon_j$, \cref{lem:cycle} identifies $S_\Psi$ with the diagonal torus
whose coordinates are constant on each connected component of the active forest.  Hence
the component partition is intrinsic, proving \cref{part: invpartition}.  If it has $r$ blocks, then
$\dim S_\Psi=r$, so $r$ is the rank of the generic stabilizer.  Finally a forest on $n$
vertices with $r$ components has $n-r$ edges, and therefore $\#\Phi=n-r$.
\end{proof}

\begin{remark} \label{rem: covers}
The weights of $T$ on the $\rad{G}$-coinvariants correspond to the \emph{covers} in $\U$,
namely those $x=(i,j)\in\U$ for which there is no $z\ne i,j$ such that
$i\homr z\extr j$ or $i\extr z\homr j$.
These are exactly the minimal elements of the containment order, i.e.\ the pairs whose row of
$M(\lambda)$ contains no parameters besides $\pm\lambda_x$.

Matrix-unit commutators show
\[
 [\nn,\uu]=\langle e_{ij}:(i,j)\text{ is not a cover}\rangle.
\]
Since the elementary unipotents $1+t e_{ab}$ generate $1+\mathfrak n$, the span of the conjugation differences $g\cdot x-x$ is precisely $[\mathfrak n,\mathfrak u]$.
If $z$ is an intermediary, $e_{ij}$ occurs as $[e_{iz},e_{zj}]$ in one of the two allowed
configurations.  Conversely every matrix unit appearing in a commutator has such an
intermediary.  Therefore the coinvariant basis is indexed exactly by covers.

In the case of success, the covers therefore are always activated. (See \cref{lem: coinvactive}.)
Hence, in this case the graph formed by the covers is acyclic.
\end{remark}

We will say that $\IM$ is strongly rigid if $\uu$ is a strong {\PVS} of $G$ in the sense of \cref{def: strongPVS}.
By \cref{lem: coinvactive}, on success the covers are always contained in the active set $\Phi$,
so $\IM$ is strongly rigid if and only if it is rigid and $\Phi$ consists precisely of the covers.

\begin{remark} \label{rem: anyfield}
The linear algebra algorithm can be tested over the prime field of $K$.

We expect that the rigidity is independent of the characteristic of $K$ (for any Dynkin quiver).
By \cref{lem:cycle} the acyclicity half of the criterion is characteristic-free, so the only possible source of
characteristic dependence is the rank test.
\end{remark}

We ran the algorithm for all incidence matrices with $n\le10$ ones, taking $M(\lambda)$ with $\lambda$ the indicator function of
the active set.
In all examples, in each step $k$ where the procedure does not abort, the first $k$ rows are a unimodular matrix, i.e.
its Smith normal form is $(I_k|0)$.
In the abort case, the $k$-th row is a linear combination of the previous rows with integer coefficients.
Thus, for every incidence matrix with at most ten ones, the procedure and rigidity verdict are independent of the characteristic.
The results are summarized in Table~\ref{tab: rigidcounts}.

\begin{table}[htbp]
\centering
\caption{Rigid and strongly rigid incidence matrices with $n\le 10$ ones.}
\label{tab: rigidcounts}
\begin{tabular}{|r|r|r|r|r|r|}
\hline
$n$ & all & rigid & strongly rigid & rigid \% & strongly rigid \%\\
\hline
1 & 1 & 1 & 1 & 100.00 & 100.00\\
2 & 4 & 4 & 4 & 100.00 & 100.00\\
3 & 24 & 24 & 24 & 100.00 & 100.00\\
4 & 196 & 194 & 184 & 98.98 & 93.88\\
5 & 2016 & 1922 & 1660 & 95.34 & 82.34\\
6 & 24976 & 21868 & 16668 & 87.56 & 66.74\\
7 & 361792 & 273156 & 180510 & 75.50 & 49.89\\
8 & 5997872 & 3634216 & 2066529 & 60.59 & 34.45\\
9 & 111969552 & 50476221 & 24654017 & 45.08 & 22.02\\
10 & 2324081728 & 722322971 & 303381713 & 31.08 & 13.05\\
\hline
\end{tabular}
\end{table}

\begin{example}[flag-dependent active sets] \label{exam: flagdep}
Consider
\[
 \IM=\begin{pmatrix}1&1&0&0\\0&0&0&1\\0&0&1&0\end{pmatrix}
\]
and label its ones by
$p_0=(1,1)$, $p_1=(1,2)$, $p_2=(2,4)$, and $p_3=(3,3)$.  Put
\[
 a=(p_1,p_2),\quad b=(p_0,p_2),\quad
 c=(p_1,p_3),\quad d=(p_0,p_3).
\]
The only strict containment relations are $a<b$ and $c<d$.  Running the algorithm in the
adapted order $(a,b,c,d)$ gives the active set $\{a,c,d\}$, whereas the adapted order
$(a,c,d,b)$ gives $\{a,c,b\}$.  Both runs succeed.  In each case the active graph is a
tree, and its connected-component partition is the same, in accordance with \cref{cor: independence}.
\end{example}

\begin{remark}
The case of general multisegments of multiplicity one can be dealt with similarly.
We chose to restrict our attention to incidence matrices since it already encapsulates the essential difficulties.
\end{remark}

\begin{remark}
Suppose that $\IM$ is rigid. Run the algorithm and obtain $\Phi$.
Consider the matrix $M(\lambda)$ where $\lambda$ is the indicator function of $\Phi$.
The nonzero entries of $M(\lambda)$ define a bipartite graph with sides $\U$ and $\V$.
Perform the following process. For each vertex in $\V$ with a unique neighbor $x$ in $\U$
remove $x$ (and the edges containing it) from the graph and repeat until this is no more possible.
Did we eliminate all vertices of $\U$ this way?
If so, this will give a certificate for the full rank of $M(\lambda)$ over any field.
\end{remark}

\subsection{The permutation case}

In the case where the incidence matrix is a permutation matrix, a precise combinatorial criterion for rigidity was given in \cite{MR3866895}.

For consistency with [ibid.], for any permutation $\sigma\in S_n$ we denote by $\IM_\sigma$ the permutation matrix of $w_0\sigma$
where $w_0$ is the longest permutation $w_0(i)=n+1-i$.

Recall that a permutation $\sigma$ in $S_n$ is called \emph{smooth} if it is $4231$ and $3412$ avoiding, i.e., if $\IM_\sigma$
does not contain the matrices
\[
\begin{pmatrix}1&0&0&0\\0&0&1&0\\0&1&0&0\\0&0&0&1\end{pmatrix},\qquad
\begin{pmatrix}0&1&0&0\\1&0&0&0\\0&0&0&1\\0&0&1&0\end{pmatrix}
\]
(the permutation matrices of $1324$ and $2143$) as $4\times 4$-submatrices.

The terminology is justified by the fact that $\sigma$ is smooth if and only if the Schubert variety pertaining to  $\sigma$
(a closed subvariety of the complete flag variety of an $n$-dimensional vector space) is smooth \cite{MR1051089}.
(We will not make use of this fact.)

\begin{theorem}[\cite{MR3866895}*{Theorem 1.2}; see also \cite{MR4847251}*{Theorem 6.1}]
$\IM_\sigma$ is rigid if and only if $\sigma$ is smooth.
\end{theorem}

We will revisit this result in the appendix and show a little more.
\begin{theorem} \label{thm: perm}
The following conditions are equivalent.
\begin{enumerate}
\item \label{part: permrigid} $\IM_\sigma$ is rigid.
\item \label{part: permstrong} $\IM_\sigma$ is strongly rigid.
\item \label{part: permsmooth} $\sigma$ is smooth.
\end{enumerate}
\end{theorem}

\begin{remark} \label{rem: LMgen}
For a general regular multisegment, the equivalent combinatorial and representation-theoretic criteria were established in \cite{MR3866895}*{Theorem~7.1};
the equivalent rigidity formulation is stated explicitly in \cite{MR4847251}*{Theorem 6.4}.
(Regular means that the beginning points are distinct and the end points are distinct, but it is not assumed that any two segments overlap.)
In principle, it should be possible to generalize \cref{thm: perm} for this case using a similar argument to the one in the appendix. However, we will not pursue this further here.
\end{remark}

\begin{remark}
The permutation matrices whose cover graph is acyclic were classified by Bousquet-M\'elou--Butler
\cite{MR2376109} in terms of pattern avoidance (with a barred pattern).
In turn, it was shown by Woo--Yong \cite{MR2422304} that this acyclicity characterizes the
locally factorial Schubert varieties.
\end{remark}

\subsection{Further remarks}

In the case of a general incidence matrix, rigidity is not preserved by deleting an arbitrary one. For instance, the matrix
\[
 \IM=\begin{pmatrix}
 0&1&1&0\\1&0&0&0\\0&0&0&1\\0&0&1&0
 \end{pmatrix}
\]
is strongly rigid: the greedy procedure succeeds with active set equal to its four cover
edges.  Deleting the entry $(1,3)$ gives
\[
 \IM'=\begin{pmatrix}
 0&1&0&0\\1&0&0&0\\0&0&0&1\\0&0&1&0
 \end{pmatrix}.
\]
The four covers of $\IM'$ form a $4$-cycle, so $\IM'$ is not rigid by
\cref{rem: covers}.

Nonetheless, some operations seem to survive, although we do not know how to prove it.

Let $\IM$ be an incidence matrix. We say that a submatrix $\IM'$ with row set $I'$ and column set $J'$ is \emph{secluded} if there is no one in $\IM$
in position $(i,j)$ where $i\in I'$ and $j\notin J'$ or $i\notin I'$ and $j\in J'$.
In other words, we can form a bipartite graph whose sides are the rows and columns of $\IM$. A secluded submatrix is one whose rows and columns are formed by
taking a union of connected components of this graph.

\begin{conjecture}
Suppose that $\IM$ is rigid and $\IM'$ is a secluded submatrix. Then $\IM'$ is also rigid.
\end{conjecture}
We verified this conjecture empirically in all cases where the number of ones in $\IM$ is $\le10$.

Ultimately, it would be very interesting to give a sharp combinatorial criterion for the rigidity of a general incidence matrix.

\appendix

\section{Appendix: smooth permutations and the results of Lapid--Minguez}\label{sec:appendix}

The purpose of the appendix is to revisit the combinatorial criterion for rigidity for smooth permutations and prove \cref{thm: perm}.
The result is not new. It was proved in \cite{MR3866895}, at least without the ``strong rigidity'' clause.
However, the proof in [ibid.] was intertwined with the representation-theoretic analogue, which was the main thrust of that paper.
So we take the opportunity to give a self-contained, clean proof which is almost entirely in terms of combinatorics of permutations.
Grosso modo, it follows the same line of argument as in [ibid.] but it is simpler in a few technical points.
Some aspects of the proof are familiar from the proof of the singular locus theorem of Schubert varieties \cites{MR1827861,MR1990570,MR1994224,MR2015302,MR1853139}.

First, we recall the symmetries of \cref{lem: symmetry}.
\begin{lemma} \label{lem: symm}
For any $\sigma\in S_n$, the rigidity/strong rigidity of $\IM_\sigma$ is preserved under $\sigma\mapsto\sigma^{-1}$ and $\sigma\mapsto w_0\sigma w_0$.
\end{lemma}

Since both permutations $4231$ and $3412$ are involutions and commute with $w_0$, the smoothness property is also invariant under the symmetries
$\sigma\mapsto\sigma^{-1}$ and $\sigma\mapsto w_0\sigma w_0$. (Of course, this can also be seen from the geometric characterization of smoothness.)

The main technical simplification in the proof here over \cite{MR3866895} is that we take advantage of all symmetries.

\subsection{Main reduction}

The key step for the proof of \cref{thm: perm} is the following result.

\begin{lemma} \label{lem: simpind}
Let $\IM$ be an incidence matrix. Suppose that the first row and the $j_1$-th column each contain a single one,
located at the position $x_1=(1,j_1)$.
Let $\IM'$ be the incidence matrix obtained by removing the first row and the $j_1$-th column.
Then, if $\IM$ is rigid, then $\IM'$ is rigid.
Conversely, suppose that the ones in the window $(i,j):j>j_1$ occur in positions $(i_k,j_k)$, $k=2,\dots,m$
(possibly with $m=1$) such that $i_1<\dots<i_m$ and $j_1<\dots<j_m$, where $(i_1,j_1)=(1,j_1)$. Then,
if $\IM'$ is rigid (resp., strongly rigid) then $\IM$ is rigid (resp., strongly rigid).
\end{lemma}

\begin{proof}
If $m=1$, the new vertex $x_1$ is isolated in the extension graph, $\U=\U'$, and strong rigidity is
unchanged.
Otherwise, denote by $z_2,\dots,z_m$ the ones in the window, and by $Y$ the remaining ones other than $x_1$
(so that the ones of $Y$ lie in columns $<j_1$ and rows $\ge2$).
By assumption, for every one $u\ne x_1$ we have $u\not\homr x_1$ (this would force $u$ to lie in the first row)
and $x_1\homr u$ if and only if $u$ lies in the window (i.e., $u\not\in Y$; here we use that the $j_1$-th column
contains no other one). In particular $(u,x_1)\notin\U$ for all $u$, and $(x_1,u)\in\U$ if and only if $u$ is in the
window. Since the deletion does not change the relative order of the coordinates of the remaining ones, we get
\[
\U=\U'\sqcup\{(x_1,z_k):k=2,\dots,m\},\qquad
\V=\V'\sqcup\{(x_1,x_1)\}\sqcup\{(x_1,u):x_1\homr u,\ u\ne x_1\},
\]
where $\U'$, $\V'$ pertain to $\IM'$.

The key observation is the following block structure of $M(\lambda)=M_\IM(\lambda)$:
\begin{enumerate}
\item \label{part: blockold} The row of $(u,v)\in\U'$ has zero entries in all columns $(x_1,\cdot)$, and its remaining entries coincide
with those of the corresponding row of $M_{\IM'}(\lambda)$. Indeed, an entry of that row in a column $(u,w)$ with
$w=x_1$ would require $u\homr x_1$, and an entry in a column $(x_1,v)$ would require $(u,x_1)\in\U$; neither occurs.
\item \label{part: blocknew} The row of $(x_1,z_k)$ has the entries $\lambda_{(z_l,z_k)}$ in the columns $(x_1,z_l)$ for those $l$
with $(z_l,z_k)\in\U$, the entries $-\lambda_{(x_1,z_l)}$ in the columns $(z_l,z_k)$ ($l<k$),
and $\pm\lambda_{(x_1,z_k)}$ in the two diagonal columns $(x_1,x_1)$ and $(z_k,z_k)$.
\end{enumerate}
Thus, in the column decomposition of $\V$ into $\V'$ and the columns $(x_1,\cdot)$, the matrix $M(\lambda)$ has the form
$\begin{pmatrix}M_{\IM'}(\lambda)&0\\ E&D\end{pmatrix}$
with the last $m-1$ rows indexed by $(x_1,z_2),\dots,(x_1,z_m)$.

For the first assertion, if $\IM$ is rigid then for generic $\lambda$
\[
\#\U=\rk M(\lambda)\le\rk M_{\IM'}(\lambda)+(m-1),
\]
so that $\rk M_{\IM'}(\lambda)\ge\#\U-(m-1)=\#\U'$, i.e., $\IM'$ is rigid.

Assume now the staircase condition of the second assertion. Then $z_2\extr z_3\extr\cdots\extr z_m$ is a chain
(both coordinates increase along it), so \emph{all} the pairs $(z_l,z_k)$, $l<k$, belong to $\U'$.
Order the columns $(x_1,\cdot)$ as $(x_1,x_1),(x_1,z_2),\dots,(x_1,z_{m-1})$
(the column $(x_1,z_m)$ vanishes identically).
By \cref{part: blocknew} above, the $(m-1)\times(m-1)$-submatrix $D$ of the rows $(x_1,z_2),\dots,(x_1,z_m)$ is lower triangular with
diagonal entries $\lambda_{(x_1,z_2)},\lambda_{(z_2,z_3)},\lambda_{(z_3,z_4)},\dots,\lambda_{(z_{m-1},z_m)}$,
hence invertible for generic $\lambda$. Therefore
$\rk M(\lambda)\ge\rk M_{\IM'}(\lambda)+(m-1)$ for generic $\lambda$, and if $\IM'$ is rigid then
$\rk M(\lambda)=\#\U'+(m-1)=\#\U$, i.e., $\IM$ is rigid.

Finally, suppose that $\IM'$ is strongly rigid. First note that
\[
\text{covers}(\IM)=\text{covers}(\IM')\sqcup\{(x_1,z_2)\}.
\]
Indeed, an element of $\U'$ has the same intermediaries in $\IM$ as in $\IM'$ (a new intermediary would have to
be comparable to $x_1$ from above, which is impossible); $(x_1,z_2)$ has no intermediary (a candidate $y$ would
satisfy $x_1\homr y$, hence lie in the window, and $y\extr z_2$ or $y\homr z_2$ with $x_1\extr y$, which fails for
$y=z_l$, $l\ge2$); and $(x_1,z_k)$, $k\ge3$, has the intermediary $z_{k-1}$. For the same reason, the consecutive
pairs $(z_{k-1},z_k)$ are covers of $\IM'$: an intermediary $y$ would satisfy $z_{k-1}\homr y$ or $z_{k-1}\extr y$,
hence lie in the window strictly between $z_{k-1}$ and $z_k$ in the chain, and there is no such one.

Now run the procedure of \cref{thm:rigidity} for $\IM$ with an adapted enumeration consisting of an adapted
enumeration of $\U'$ followed by $(x_1,z_2),\dots,(x_1,z_m)$. (This is adapted: no element of $\U'$ contains any
$(x_1,z_k)$, since this would require its first index $u$ to satisfy $u\homr x_1$.) By \cref{part: blockold} of the block structure,
the first $\#\U'$ steps replicate the run for $\IM'$, which succeeds with
$\Phi'=\text{covers}(\IM')$. The step $(x_1,z_2)$: its row vanishes at test time (all its potential parameters are
inactive or equal to $\lambda_{(x_1,z_2)}$ itself), so the step activates; the edge $\{x_1,z_2\}$ involves the new
vertex $x_1$ and creates no cycle. The step $(x_1,z_k)$, $k\ge3$: since $(z_{k-1},z_k)\in\Phi'$, the row has the
nonzero entry $\lambda_{(z_{k-1},z_k)}=1$ in the column $(x_1,z_{k-1})$, whereas all previously processed rows vanish
in that column (the rows of $\U'$ by \cref{part: blockold}, and the rows $(x_1,z_l)$, $l<k$, because their entries in the columns
$(x_1,z_j)$ occur only for $j<l\le k-1$). Hence the row is independent from the previous ones and the step does not
activate. The procedure therefore succeeds with $\Phi=\Phi'\sqcup\{(x_1,z_2)\}=\text{covers}(\IM)$, so $\IM$ is
strongly rigid.
\end{proof}

\begin{remark}
The first assertion holds, with the same proof, without any assumption on the $j_1$-th column, provided $\IM'$ is
interpreted as $\IM$ with the entry $x_1$ (and the emptied first row) removed. We will not need this.
\end{remark}

\subsection{Smooth implies strongly rigid}

It is easy to deduce from the second part of \cref{lem: simpind} that $\IM_\sigma$ is strongly rigid for every smooth $\sigma$ by induction on $n$.
The induction step is based on the following observation.

\begin{lemma} \label{lem: red1row}
Suppose that $\sigma\in S_n$ is smooth. Then, the following condition is satisfied for $\tau=\sigma$ or $\tau=\sigma^{-1}$ (or both):
for any $i<j$ such that $\tau(i)<\tau(1)$ and $\tau(j)<\tau(1)$ we have $\tau(j)<\tau(i)$.
\end{lemma}

\begin{proof}
Suppose on the contrary that the condition fails for both $\tau=\sigma$ and $\tau=\sigma^{-1}$. The failure for
$\sigma$ means that there exist $1<i<j$ with $\sigma(i)<\sigma(j)<\sigma(1)$. The failure for $\sigma^{-1}$,
read in terms of positions of values, means that there exist $p<q<w:=\sigma^{-1}(1)$ with $\sigma(p)<\sigma(q)$.
Note that $w\ne j$ since $\sigma(j)>\sigma(i)\ge1$. We exhibit a forbidden pattern in each of the following
exhaustive cases, contradicting the smoothness of $\sigma$.
\begin{enumerate}
\item If $w>j$, then the positions $(1,i,j,w)$, with values $(\sigma(1),\sigma(i),\sigma(j),1)$, form an occurrence
of $4231$.
\item If $w<j$ and $\sigma(q)>\sigma(1)$ (so that $q>1$), then the positions $(1,q,w,j)$, with values
$(\sigma(1),\sigma(q),1,\sigma(j))$, form an occurrence of $3412$.
\item If $w<j$ and $\sigma(q)<\sigma(1)$, then $p>1$ (since $\sigma(p)<\sigma(q)$) and the positions $(1,p,q,w)$,
with values $(\sigma(1),\sigma(p),\sigma(q),1)$, form an occurrence of $4231$.\qedhere
\end{enumerate}
\end{proof}

Denote by $\partial\sigma$ the permutation in $S_{n-1}$ obtained from $\sigma$ by removing $(1,\sigma(1))$ and keeping the order,
so that $\IM_{\partial\sigma}$ is obtained from $\IM_\sigma$ by deleting the first row and the $(w_0\sigma(1))$-th column.
Clearly, if $\sigma$ is smooth, then $\partial\sigma$ is also smooth.

\begin{proposition} \label{prop: smoothstrong}
If $\sigma$ is smooth, then $\IM_\sigma$ is strongly rigid.
\end{proposition}

\begin{proof}
By induction on $n$, the case $n=1$ being trivial. Let $\sigma\in S_n$ be smooth, $n\ge2$, and let $\tau\in\{\sigma,\sigma^{-1}\}$
be as in \cref{lem: red1row}. By \cref{lem: symm} we may replace $\sigma$ by $\tau$ and assume $\tau=\sigma$.
The matrix $\IM_\sigma$ is a permutation matrix, so the hypotheses of \cref{lem: simpind} on the first row
and the $j_1$-th column ($j_1=w_0\sigma(1)$) hold. The ones in the window are the $(i,w_0\sigma(i))$ with $i\ge2$ and
$\sigma(i)<\sigma(1)$, and the condition of \cref{lem: red1row} for $\tau=\sigma$ says precisely that their columns
increase along with their rows, i.e., the staircase hypothesis of \cref{lem: simpind} holds.
Since $\partial\sigma$ is smooth, $\IM_{\partial\sigma}=\IM_\sigma'$ is strongly rigid by the induction hypothesis,
and \cref{lem: simpind} gives the claim.
\end{proof}

For smooth $\sigma$ we can also describe the partition of the ones of $\IM_\sigma$ into the connected components
of the graph formed by the covers (which, by strong rigidity, is the partition $P(\Phi)$).
Recall that $\sigma\in S_n$ is \emph{decomposable} if $\sigma(\{1,\dots,k\})=\{1,\dots,k\}$ for some $0<k<n$;
correspondingly, $\sigma$ is uniquely a direct sum $\sigma=\sigma_1\oplus\dots\oplus\sigma_r$ of indecomposable
permutations. The resulting partition of $\{1,\dots,n\}$ into consecutive intervals induces a partition of the ones
of $\IM_\sigma$, whose blocks occupy anti-diagonal blocks of $\IM_\sigma$.

\begin{proposition} \label{prop: covcomp}
Let $\sigma$ be smooth. Then the connected components of the cover graph of $\IM_\sigma$ are exactly the blocks
corresponding to the indecomposable summands of $\sigma$.
\end{proposition}

\begin{proof}
First, for any $\sigma$, two ones of $\IM_\sigma$ belonging to different summands are incomparable: if the
positions $i<j$ belong to different summands then $\sigma(i)<\sigma(j)$, so the rows of the two ones increase while
their columns decrease. Consequently, $\U$ is the disjoint union of the sets $\U$ of the matrices
$\IM_{\sigma_1},\dots,\IM_{\sigma_r}$, an intermediary of a pair of ones in one summand lies in the same summand
(being comparable to a one of it), and the cover graph of $\IM_\sigma$ is the disjoint union of the cover graphs of
the $\IM_{\sigma_i}$. In particular, the components refine the blocks, and it remains to show that the cover graph
of $\IM_\sigma$ is connected whenever $\sigma$ is smooth and indecomposable.

We argue by induction on $n$, the case $n=1$ being trivial. As in the proof of \cref{prop: smoothstrong}, using \cref{lem: symm} (which preserves the cover graph, while
$\sigma\mapsto\sigma^{-1}$ preserves decomposability) we may assume that the condition of \cref{lem: red1row}
holds for $\tau=\sigma$. Since $\sigma$ is indecomposable, $t:=\sigma(1)>1$, so the window of $x_1=(1,n+1-t)$ is
nonempty; let $z_2$ be its one with the smallest row index. As was shown in the course of the proof of \cref{lem: simpind},
\[
\text{covers}(\IM_\sigma)=\text{covers}(\IM_{\partial\sigma})\sqcup\{(x_1,z_2)\}.
\]
Thus, the cover graph of $\IM_\sigma$ is obtained from that of $\IM_{\partial\sigma}$ by attaching the vertex $x_1$
along the edge $\{x_1,z_2\}$, and by the induction hypothesis it suffices to show that $\partial\sigma$ is
indecomposable.

Suppose on the contrary that the first $k$ entries of $\partial\sigma$, namely $\sigma(2),\dots,\sigma(k+1)$ with
$1\le k\le n-2$, form the set of the $k$ smallest elements of $\{1,\dots,n\}\setminus\{t\}$. If $t\le k$, this set
is $\{1,\dots,k+1\}\setminus\{t\}$, so $\sigma(\{1,\dots,k+1\})=\{1,\dots,k+1\}$ with $k+1<n$, contradicting the
indecomposability of $\sigma$. Hence $t>k$ and $\{\sigma(2),\dots,\sigma(k+1)\}=\{1,\dots,k\}$. These values are
smaller than $t$, so by the condition of \cref{lem: red1row} they appear in decreasing order, i.e.,
$(\sigma(2),\dots,\sigma(k+1))=(k,k-1,\dots,1)$. Moreover, there is no value $v$ with $k<v<t$: such $v$ would occur
after the position $k+1$ (the earlier positions being taken), and then the values smaller than $t$ would read
$(k,\dots,1,\dots,v,\dots)$, again contradicting \cref{lem: red1row}. Therefore $t=k+1$ and
$\sigma(\{1,\dots,k+1\})=\{1,\dots,k+1\}$ with $k+1<n$, contradicting the indecomposability of $\sigma$ once more.
\end{proof}

\begin{remark}
Combining \cref{prop: covcomp} with strong rigidity ($\Phi=\text{covers}$) and the last part of
\cref{thm:quiver}, we conclude that for smooth $\sigma$ the decomposition of a generic point into
$\Pi$-indecomposables matches the decomposition of $\sigma$ into indecomposable permutations.
A closely related statement was proved in \cite{MR4847251}*{Proposition 6.7} but the argument here is more conceptual.
\end{remark}

\subsection{Minimal violators}
For the proof of \cref{thm: perm}, it remains to show that if $\sigma$ is not smooth, then $\IM_\sigma$ is not rigid.
If we knew \emph{a priori} that rigidity is closed under deleting an entry of a permutation, then we would only need to show the non-rigidity
for the permutations $3412$ and $4231$ themselves.
Of course, this hereditary property would ultimately follow from \cref{thm: perm}, but we do not know how to prove it in general without assuming
\cref{thm: perm}.

Fortunately, the first (easier) part of \cref{lem: simpind} handles a special case of this hereditary property and this already provides a substantial reduction of the problem.
Namely, the Lemma implies that if $\IM_{\partial\sigma}$ is non-rigid, then $\IM_\sigma$ is also non-rigid.
Together with the symmetries of \cref{lem: symm} we can weaken the protasis and thereby strengthen the statement to: if $\IM_{\partial\tau}$ is non-rigid for
$\tau=\sigma$, $\sigma^{-1}$, $w_0\sigma w_0$ or $w_0\sigma^{-1}w_0$, then $\IM_\sigma$ is non-rigid.
Note that these four operations amount to deleting from $\sigma$ (and renumbering) the first entry, the last entry, the entry
of value $1$, or the entry of value $n$, which corresponds to deleting from $\IM_\sigma$ the first row, the last row, the
last column, or the first column, respectively (together with the row or column of the unique one contained in it).

This motivates the following definition: we say that $\sigma$ is a \emph{minimal violator} if $\sigma$ is not smooth but the four permutations
$\partial\tau$, $\tau=\sigma,\sigma^{-1},w_0\sigma w_0,w_0\sigma^{-1}w_0$ (in $S_{n-1}$) are all smooth.
Note that unlike in the previous implication, it is advantageous here to use all four symmetries.

By the above,
\begin{equation} \label{quote: minviolsuffice}
\begin{aligned}
&\text{in order to prove $\IM_\sigma$ is not rigid for any non-smooth $\sigma$ (thereby completing}\\
&\text{the proof of \cref{thm: perm}) it is enough to do so for the minimal violators.}
\end{aligned}
\end{equation}

Fortunately, it is possible to characterize the minimal violators rather explicitly.

Given a permutation $\sigma\in S_m$ and integers $k_1,\dots,k_m\ge1$, the \emph{thickening} $\tilde\sigma$ of $\sigma$ with parameters $k_1,\dots,k_m$
is the permutation of size $m+\sum(k_i-1)$ such that $\IM_{\tilde\sigma}$ is obtained from $\IM_\sigma$ by replacing the entry in $i$-th row by the identity matrix of size $k_i$.
For instance, the thickening of $45312$ with parameters $(1,1,2,1,1)$ is $564312$:
\[
\IM_{45312}=
\begin{pmatrix}0&1&0&0&0\\1&0&0&0&0\\0&0&1&0&0\\0&0&0&0&1\\0&0&0&1&0\end{pmatrix}
\rightsquigarrow
\begin{pmatrix}0&1&0&0&0&0\\1&0&0&0&0&0\\0&0&1&0&0&0\\0&0&0&1&0&0\\0&0&0&0&0&1\\0&0&0&0&1&0\end{pmatrix}
=\IM_{564312},
\]
while the thickenings of $(m,2,\dots,m-1,1)$ with $k_1=k_m=1$ are the permutations
\begin{equation} \label{eq: famI}
(k,\ D_1,\ D_2,\ \dots,\ D_t,\ 1)\in S_k,\qquad t\ge1,
\end{equation}
whose middle part is the partition of $\{2,\dots,k-1\}$ into $t$ nonempty \emph{decreasing} runs of consecutive
integers $D_1=(a_1,a_1-1,\dots,2)$, $D_2=(a_2,\dots,a_1+1)$, \dots, $D_t=(k-1,\dots,a_{t-1}+1)$,
arranged in increasing order of their values.

\begin{lemma} \label{lem: minviol}
The minimal violators are precisely the following permutations.
\begin{enumerate}
\item \label{part: mvI} The thickenings, with $k_1=k_m=1$, of $(m,2,\dots,m-1,1)$, $m\ge4$, i.e., the permutations \cref{eq: famI} with $t\ge2$ runs.
\item \label{part: mvII} The permutations
\[
\tau_{r,s}:=(r+1,\,r-1,r-2,\dots,2,\,k,\,1,\,k-1,k-2,\dots,r+2,\,r)\in S_k,\quad k=r+s,\ r,s\ge2
\]
(for $r=s=2$ this is $3412$).
\item \label{part: mvIII} The permutations
\[
\upsilon_r:=(r+1,\ r+2,\ r,r-1,\dots,3,\ 1,\ 2)\in S_{r+2},\qquad r\ge3,
\]
i.e., the thickenings of $45312=\upsilon_3$ in which only the middle entry is thickened.
\end{enumerate}
\end{lemma}

\begin{remark}
The permutations in \cref{part: mvII,part: mvIII} and the two-run case of \labelcref{part: mvI} are exactly the permutations governing the generic
singularities of Schubert varieties of type $A$, determined independently in \cites{MR1990570,MR1994224,MR2015302,MR1853139}
(they are denoted $\tau^{(2)}_{r,s}$, $\tau^{(3)}_{r,2}$, $\tau^{(1)}_{r,s}$ in \cite{MR3866895}*{\S5.2}).
In contrast, family \labelcref{part: mvI} contains all multi-run thickenings, not only the two-run ones.
Note also that the list is not the one guessed at first sight: for instance, the ``two-sided'' thickenings of $3412$
and $45312$ are non-smooth but are not minimal violators, as suitable corner deletions keep them non-smooth.
\end{remark}

\begin{proof}[Proof of \cref{lem: minviol}]
The proof is routine, if tedious.

It is convenient to argue in terms of the matrix $\IM=\IM_\sigma$, whose ones are $(i,\pi(i))$, $i=1,\dots,n$,
where $\pi=w_0\sigma$; thus $\sigma$ is smooth if and only if $\IM$ contains no occurrence of (the permutation
matrices of) $1324$ or $2143$, and in the list of the lemma the three families read, in terms of $\pi$:
\begin{align}
&\pi=(1,\ B_1,\dots,B_t,\ n),\label{eq: famIpi}\\
&\pi=(s,\ s+2,s+3,\dots,s+r-1,\ 1,\ n,\ 2,3,\dots,s-1,\ s+1),\label{eq: famIIpi}\\
&\pi=(2,\ 1,\ 3,4,\dots,n-2,\ n,\ n-1),\label{eq: famIIIpi}
\end{align}
where in \cref{eq: famIpi} the middle part is the partition of $\{2,\dots,n-1\}$ into $t$ \emph{increasing} runs of
consecutive integers arranged in decreasing order of their values (the three equations corresponding to the three
families of the lemma, in order).
Call the ones of $\IM$ in the first row, the last row, the first column and the last column the \emph{border} ones;
they are $f=(1,u)$, $l=(n,v)$, $c=(a,1)$ and $d=(b,n)$, where $u=\pi(1)$, $v=\pi(n)$, $a=\pi^{-1}(1)$, $b=\pi^{-1}(n)$.
Since deleting a border one (with its row and column) preserves any occurrence not containing it, and destroys any
occurrence containing it, we have:
\emph{$\sigma$ is a minimal violator if and only if $\IM$ contains an occurrence of $1324$ or $2143$, and every such
occurrence contains all four border ones.}

\emph{Step 1: coincidences among the border ones.}
The possible coincidences are $f=c$ ($u=1$), $f=d$ ($u=n$), $l=c$ ($v=1$), $l=d$ ($v=n$).
If $u=n$, then every occurrence contains $f$, i.e., its first letter is its largest one; but the patterns $1324$ and
$2143$ begin with their first and second smallest letters, respectively. Hence there is no occurrence at all,
contradicting non-smoothness. Similarly $v=1$ is impossible (the patterns do not end with their smallest letter).
Next suppose $u=1$ but $v\ne n$. Then $f=c$ and the border ones are $f$, $d=(b,n)$, $l=(n,v)$ with $v<n$; an
occurrence contains all three, and since its first letter is the smallest one ($1$ at position $1$), its pattern
begins with its minimal letter, so it must be $1324$; but then its last letter is its largest one, whereas the letter
at the last position is $v<n$, and $n$ occurs inside. This is a contradiction; hence $u=1$ implies $v=n$, and
symmetrically $v=n$ implies $u=1$. We are left with two cases: either $u=1$ and $v=n$, or $1<u,v<n$ and the four
border ones are distinct.

\emph{Step 2: the case $u=1$, $v=n$.}
Here $f=c=(1,1)$ and $l=d=(n,n)$. Let $W$ be the word $\pi(2),\dots,\pi(n-1)$.
For interior positions $i<j$, the quadruple $(1,i,j,n)$ is an occurrence (necessarily of $1324$) if and only if
$\pi(i)>\pi(j)$; such occurrences contain all border ones. Any other potential occurrence omits a border one and is
therefore forbidden:
\begin{itemize}
\item $(i,j,k,n)$, $1<i<j<k<n$: an occurrence of $1324$ if and only if $(\pi(i),\pi(j),\pi(k))$ forms a
$132$-pattern (an occurrence of $2143$ is impossible since the largest letter of $2143$ is its third one).
Hence $W$ must avoid $132$.
\item $(1,i,j,k)$: an occurrence of $1324$ if and only if $(\pi(i),\pi(j),\pi(k))$ forms a $213$-pattern
($2143$ is impossible since the second letter of $2143$ is its smallest one). Hence $W$ must avoid $213$.
\item occurrences within the interior contain a $132$- or $213$-pattern and are excluded by the above.
\end{itemize}
Thus, the case at hand holds if and only if $W$ avoids $132$ and $213$ and is not increasing.
We claim that $W$ avoids $132$ and $213$ if and only if $W$ is a concatenation of increasing runs of consecutive
integers, arranged in decreasing order of their values (so that non-increasing means $t\ge2$ runs, giving exactly the
family \cref{eq: famIpi}). Indeed, if $W$ has this form, then any ascent of $W$ lies within a single run, and both
patterns require an ascent $(x,y)$ together with a letter strictly between $x$ and $y$ located outside the positions
between them, which is impossible for a run of consecutive integers. Conversely, suppose $W$ avoids both patterns and
let $d$ be the position of the largest letter $t_0$ of $W$. Avoidance of $132$ for the triples $(i,d,k)$, $i<d<k$,
forces every letter after $d$ to be smaller than every letter before $d$; avoidance of $213$ for the triples
$(i,i',d)$ forces the prefix up to $d$ to be increasing. Hence this prefix consists of the largest letters in
increasing order, i.e., it is an increasing run of consecutive integers ending with $t_0$, and we conclude by
induction on the length applied to the remaining suffix.

\emph{Step 3: the case of distinct border ones.}
Every occurrence contains the four distinct border ones, hence consists exactly of them. The positions of $f,c,d,l$
are $(1,a,b,n)$ or $(1,b,a,n)$ in increasing order. If $b<a$ then the letter $n$ precedes the letter $1$ at the
second and third positions of the occurrence; but in $1324$ and $2143$ the largest letter does not precede the
smallest one. So there is no occurrence, a contradiction; hence $a<b$. The unique candidate occurrence
$(1,a,b,n)$ has the letters $(u,1,n,v)$, which form an occurrence (necessarily of $2143$) if and only if $u<v$.
Non-smoothness thus forces
\[
1<u<v<n,\qquad 1<a<b<n,
\]
and the requirement is that \emph{no other} quadruple is an occurrence. Let
$P$, $Q$, $R$ be the sets of non-border positions $i$ with $i<a$, $a<i<b$, $b<i$, respectively.
Forbidding the following quadruples (none of which contains all four border ones) yields, in turn:
\begin{itemize}
\item $(i,a,b,n)$, $i\in P$: letters $(\pi(i),1,n,v)$, an occurrence of $2143$ if and only if $\pi(i)<v$.
Hence $\pi(i)>v$ on $P$.
\item $(1,a,b,j)$, $j\in R$: letters $(u,1,n,\pi(j))$, an occurrence of $2143$ if and only if $\pi(j)>u$.
Hence $\pi(j)<u$ on $R$.
\item $(1,a,j,n)$, $j\in Q\cup R$: letters $(u,1,\pi(j),v)$, an occurrence of $2143$ if and only if $\pi(j)>v$.
Hence $\pi(j)<v$ for $j>a$; together with the previous item, $u<\pi(j)<v$ on $Q$.
\item $(1,i,j,b)$, $i<j$ in $P$: letters $(u,\pi(i),\pi(j),n)$, an occurrence of $1324$ if and only if
$u<\pi(j)<\pi(i)$. Since $\pi(i),\pi(j)>v>u$, this forces $\pi$ to be increasing on $P$.
Similarly, $(1,i,j,n)$ for $i<j$ in $Q$ and $(a,i,j,n)$ for $i<j$ in $R$ force $\pi$ to be increasing on $Q$ and
on $R$.
\item $(1,i,j,b)$, $i\in P$, $j\in Q$: letters $(u,\pi(i),\pi(j),n)$ with $u<\pi(j)<\pi(i)<n$ --- always an
occurrence of $1324$. Hence $P$ and $Q$ cannot both be nonempty. Similarly, $(a,i,j,n)$ for $i\in Q$, $j\in R$
has letters $(1,\pi(i),\pi(j),v)$ with $1<\pi(j)<\pi(i)<v$, so $Q$ and $R$ cannot both be nonempty.
\end{itemize}
Counting values yields $\#P=n-1-v$, $\#Q=v-u-1$, $\#R=u-2$.
If $Q=\emptyset$ then $v=u+1$, $b=a+1$, and the constraints above determine $\pi$ completely:
$\pi$ is given by \cref{eq: famIIpi} with $s=u$ and $r=n-u$ (the case $P=\emptyset$, resp.\ $R=\emptyset$,
corresponding to $r=2$, resp.\ $s=2$). If $Q\ne\emptyset$ then $P=R=\emptyset$, so $u=2$, $v=n-1$, $a=2$, $b=n-1$,
and $\pi$ is given by \cref{eq: famIIIpi}.

\emph{Step 4: sufficiency.}
In Step 2 it was already shown that the permutations \cref{eq: famIpi} with $t\ge2$ are minimal violators.
For \cref{eq: famIIIpi}, the only inversions of $\pi$ are at the positions $(1,2)$ and $(n-1,n)$; an occurrence of
$1324$ uses an inversion with a letter before it and a letter after it, which is impossible here, while an occurrence
of $2143$ uses two disjoint inversions, so the unique occurrence is $(1,2,n-1,n)$, which consists exactly of the four
border ones. For \cref{eq: famIIpi}, write $\pi=(s,\ P,\ 1,\ n,\ R,\ s+1)$, where $P$ (occupying the positions $2,\dots,a-1$,
$a=r$) and $R$ (occupying the positions $b+1,\dots,n-1$, $b=r+1$) are the increasing words with letters
$s+2,\dots,s+r-1$ and $2,\dots,s-1$, respectively. Abusively denoting positions by the blocks they belong to,
the inversions of $\pi$ are: $(x,a)$ with $x\in\{1\}\cup P$; $(x,y)$ with $x\in P$, $y\in R\cup\{n\}$;
$(1,y)$ with $y\in R$; and $(b,y)$ with $y\in R\cup\{n\}$.
An occurrence of $2143$ consists of two disjoint inversions $(i_1,i_2)$, $(i_3,i_4)$ with $i_2<i_3$ and
$\pi(i_2)<\pi(i_1)<\pi(i_4)<\pi(i_3)$. From the catalogue, an inversion $(i_1,i_2)$ with $i_2<i_3\le b$ must have
$i_2=a$, so $\pi(i_2)=1$ and $\pi(i_1)\in\{s\}\cup P$; and then $i_3=b$ ($i_3=a$ is excluded, and there is no
inversion starting in $R$). If $i_4\in R$ then $\pi(i_4)\le s-1<s\le\pi(i_1)$, a contradiction; so $i_4=n$ and
$\pi(i_4)=s+1$, forcing $\pi(i_1)=s$, i.e., $i_1=1$. Thus the unique occurrence of $2143$ is $(1,a,b,n)$.
An occurrence of $1324$ consists of an inversion $(i_2,i_3)$ together with positions $i_1<i_2$, $i_4>i_3$ such that
$\pi(i_1)<\pi(i_3)$ and $\pi(i_4)>\pi(i_2)$. Running over the catalogue: if $i_3=a$ then $\pi(i_1)<1$, which is
impossible; if $i_2\in\{b\}$ then $\pi(i_4)>n$, impossible; if $i_2\in P$ and $i_3\in R\cup\{n\}$ then either
$i_3=n$ is the last position (no room for $i_4$) or $\pi(i_4)>\pi(i_2)\ge s+2$ with $i_4$ after a position of $R$,
whereas all letters there are at most $s+1$, impossible; and if $i_2=1$, $i_3\in R$, then there is no room for
$i_1$. Hence there is no occurrence of $1324$. Since every occurrence contains all four border
ones, all four deletions are smooth, while $\sigma$ itself is not.
\end{proof}

\subsection{End of proof}

It remains to prove that the minimal violators are not rigid. We use two general criteria, valid for an arbitrary incidence matrix $\IM$.

By \cref{rem: covers}, if the covers of $\U$ (viewed as edges on $\{1,\dots,n\}$) contain a cycle, then $\IM$ is not rigid.
For permutation matrices, this criterion propagates to thickenings.

\begin{lemma} \label{lem: thickcycle}
Let $\widetilde\sigma$ be a thickening of $\sigma$.  The cover graph of
$\IM_{\widetilde\sigma}$ contains a cycle if and only if the cover graph of
$\IM_\sigma$ does. 
\end{lemma}

\begin{proof}
It is enough to consider one duplication, i.e., replacing a one $z=(a,b)$ with the identity block consisting of $z_1=(a,b)$ and $z_2=(a+1,b+1)$ (and shifting the
subsequent rows and columns). Since $\IM_\sigma$ is a permutation matrix, no one of $\IM_{\tilde\sigma}$ other than
$z_1,z_2$ lies in the rows $a,a+1$ or the columns $b,b+1$. Hence, for every such one $w$ we have
$w\homr z_i\iff w\homr z$, $z_i\homr w\iff z\homr w$, and likewise for $\extr$ ($i=1,2$; recall that for a
permutation matrix $w\homr z$, $w\ne z$ implies $w\extr z$); in addition $z_1\extr z_2$. It follows that
\begin{itemize}
\item for $u,v\notin\{z_1,z_2\}$, the pair $(u,v)$ is a cover of $\IM_{\tilde\sigma}$ if and only if it is a cover of
$\IM_\sigma$ (the new candidate intermediaries $z_1,z_2$ behave exactly like $z$);
\item $(w,z_1)$ is a cover if and only if $(w,z)$ is (the only new candidate intermediary is $z_2$, and
$z_2\extr z_1$, $z_2\homr z_1$ both fail); similarly $(z_2,w)$ is a cover if and only if $(z,w)$ is;
\item $(w,z_2)$ and $(z_1,w)$ are never covers ($z_1$, resp.\ $z_2$, is an intermediary);
\item $(z_1,z_2)$ is always a cover (an intermediary $y$ would satisfy $z\homr y$ and $y\homr z$, hence $y=z$).
\end{itemize}
Thus, the new cover graph is obtained by replacing a vertex $z$ by an edge $z_1-z_2$, attaching
all incoming cover edges at $z_1$ and all outgoing cover edges at $z_2$.  A cycle in the
old graph therefore lifts to a cycle in the new one.  Conversely, contracting the edge
$z_1-z_2$ maps every cycle in the new graph to a closed walk in the old graph, which
contains a cycle.
\end{proof}

\begin{corollary} \label{cor: families12}
The thickenings of $(m,2,\dots,m-1,1)$, $m\ge4$, as well as the permutations $\tau_{r,s}$ of \cref{lem: minviol}, are not rigid.
\end{corollary}

\begin{proof}
Note that $\tau_{r,s}$ is a thickening of $\tau_{r_0,s_0}$, where $r_0=\min(r,3)$ and $s_0=\min(s,3)$ and $\tau_{r,s}=w_0\tau_{s,r}w_0$.

Therefore, by \cref{lem: thickcycle} and symmetries it suffices to exhibit a cycle among the covers of $\IM_\tau$ where $\tau$ is
$(m,2,\dots,m-1,1)$, $m\ge4$, $\tau_{2,2}=3412$, $\tau_{3,2}=42513$ and $\tau_{3,3}=426153$. We have the following cycles in the cover graph
(denoting the vertices by their one-positions in $\IM_\tau$):
\begin{align*}
(m,2,\dots,m-1,1):&\quad (1,1)-(2,m-1)-(m,m)-(m-1,2)-(1,1),\\
3412:&\quad (1,2)-(3,4)-(2,1)-(4,3)-(1,2),\\
42513:&\quad (1,2)-(2,4)-(4,5)-(3,1)-(5,3)-(1,2),\\
426153:&\quad (1,3)-(2,5)-(4,6)-(3,1)-(5,2)-(6,4)-(1,3).
\end{align*}
\end{proof}

When the covering criterion does not apply we will use the following criterion.

\begin{proposition} \label{prop: covertree}
Suppose that the covers of $\U$, viewed as edges on the ones of $\IM$, form a spanning tree.
Then the following are equivalent:
\begin{enumerate}
\item \label{part: treerigid} $\IM$ is rigid;
\item \label{part: treestrong} $\IM$ is strongly rigid;
\item \label{part: treerank} $\rk M(\lambda)=\#\U$ for one (equivalently, every) $\lambda$ whose support is the set of covers.
\end{enumerate}
\end{proposition}

\begin{proof}
\labelcref{part: treerigid}$\Rightarrow$\labelcref{part: treestrong}: in a successful run of the procedure of \cref{thm:rigidity} the
the active set $\Phi$ contains the covers by \cref{rem: covers} and \cref{lem: coinvactive} and is acyclic, hence
coincides with the spanning tree. \labelcref{part: treestrong}$\Rightarrow$\labelcref{part: treerank} is \cref{thm:rigidity}\labelcref{part: rigforest} with
$\Phi=\mathrm{covers}$, and \labelcref{part: treerank}$\Rightarrow$\labelcref{part: treerigid} is clear, since full rank at one point implies generic full rank.
\end{proof}

\begin{corollary}
The permutations $\upsilon_r$ of \cref{lem: minviol} are not rigid.
\end{corollary}

\begin{proof}
The ones are $x_1=(1,2)$, $B=(2,1)$, $M_i=(2+i,2+i)$ ($1\le i\le q$, $q=r-2\ge1$), $T=(q+3,q+4)$
and $z=(q+4,q+3)$. Here the covers are the consecutive pairs $(M_i,M_{i+1})$ together with $(x_1,M_1)$, $(B,M_1)$,
$(M_q,T)$ and $(M_q,z)$; they form a tree, so the cover criterion does not apply. Instead we use \cref{prop: covertree}.

The only cover into $T$ (resp.\ $z$) is $(M_q,T)$ (resp.\ $(M_q,z)$), and the only cover out of $x_1$ (resp.\ $B$) is $(x_1,M_1)$
(resp.\ $(B,M_1)$). Hence, for $\lambda$ supported on the covers, the rows of $M(\lambda)$
indexed by $(x_1,T)$, $(x_1,z)$, $(B,T)$, $(B,z)$ are (writing $e_v$, $v\in\V$, for the standard
basis of the column space)
\begin{align*}
\operatorname{row}(x_1,T)&=\lambda_{(M_q,T)}e_{(x_1,M_q)}-\lambda_{(x_1,M_1)}e_{(M_1,T)},\\
\operatorname{row}(x_1,z)&=\lambda_{(M_q,z)}e_{(x_1,M_q)}-\lambda_{(x_1,M_1)}e_{(M_1,z)},\\
\operatorname{row}(B,T)&=\lambda_{(M_q,T)}e_{(B,M_q)}-\lambda_{(B,M_1)}e_{(M_1,T)},\\
\operatorname{row}(B,z)&=\lambda_{(M_q,z)}e_{(B,M_q)}-\lambda_{(B,M_1)}e_{(M_1,z)},
\end{align*}
which satisfy the identity
\begin{multline*}
\lambda_{(M_q,z)}\lambda_{(B,M_1)}\operatorname{row}(x_1,T)
-\lambda_{(M_q,T)}\lambda_{(B,M_1)}\operatorname{row}(x_1,z)\\
-\lambda_{(M_q,z)}\lambda_{(x_1,M_1)}\operatorname{row}(B,T)
+\lambda_{(M_q,T)}\lambda_{(x_1,M_1)}\operatorname{row}(B,z)=0.
\end{multline*}
Hence $\rk M(\lambda)<\#\U$ and $\IM$ is not rigid by \cref{prop: covertree}.
\end{proof}

Together with \cref{cor: families12} we conclude from \cref{lem: minviol} that the minimal violators are not rigid.
By \cref{quote: minviolsuffice} this yields the direction \labelcref{part: permrigid}$\Rightarrow$\labelcref{part: permsmooth} of \cref{thm: perm}.
Since \labelcref{part: permsmooth}$\Rightarrow$\labelcref{part: permstrong} is \cref{prop: smoothstrong} and \labelcref{part: permstrong}$\Rightarrow$\labelcref{part: permrigid} holds trivially this concludes the proof of \cref{thm: perm}.

\section*{Acknowledgments}

The author is deeply indebted to Alberto M\'{\i}nguez for many discussions during the earlier stage of the project
and for allowing me to include some of the ideas which came up in the discussions.
The author would also like to thank Shachar Carmeli for a number of helpful conversations.

\section*{Declaration of generative AI and AI-assisted technologies in
the manuscript preparation process}

During the preparation and revision of this manuscript, the author used
ChatGPT (OpenAI) and Claude Opus 5 (Anthropic) to assist with language and exposition, to identify
possible internal inconsistencies and cross-reference issues, to flag
questions concerning bibliographic references and attribution, and to create the code for producing Table~\ref{tab: rigidcounts}. The author
independently checked the mathematical arguments, verified any
AI-assisted bibliographic suggestions against the cited sources, reviewed
and edited all AI-assisted suggestions, and takes full responsibility for
the content of the article.

\def\cprime{$'$}


\begin{bibdiv}
\begin{biblist}

\bib{MR4847234}{article}{
      author={Aizenbud, Avraham},
      author={Lapid, Erez},
       title={A binary operation on irreducible components of {L}usztig's
  nilpotent varieties {I}: definition and properties},
        date={2025},
        ISSN={1558-8599,1558-8602},
     journal={Pure Appl. Math. Q.},
      volume={21},
      number={1},
       pages={5\ndash 41},
         url={https://doi.org/10.4310/pamq.241203024730},
      review={\MR{4847234}},
}

\bib{MR1990570}{article}{
      author={Billey, Sara~C.},
      author={Warrington, Gregory~S.},
       title={Maximal singular loci of {S}chubert varieties in {${\rm
  SL}(n)/B$}},
        date={2003},
        ISSN={0002-9947},
     journal={Trans. Amer. Math. Soc.},
      volume={355},
      number={10},
       pages={3915\ndash 3945},
         url={http://dx.doi.org/10.1090/S0002-9947-03-03019-8},
      review={\MR{1990570}},
}

\bib{MR1102012}{book}{
      author={Borel, Armand},
       title={Linear algebraic groups},
     edition={Second},
      series={Graduate Texts in Mathematics},
   publisher={Springer-Verlag},
     address={New York},
        date={1991},
      volume={126},
        ISBN={0-387-97370-2},
      review={\MR{MR1102012 (92d:20001)}},
}

\bib{MR2376109}{article}{
      author={Bousquet-M\'elou, Mireille},
      author={Butler, Steve},
       title={Forest-like permutations},
        date={2007},
        ISSN={0218-0006},
     journal={Ann. Comb.},
      volume={11},
      number={3-4},
       pages={335\ndash 354},
         url={http://dx.doi.org/10.1007/s00026-007-0322-1},
      review={\MR{2376109}},
}

\bib{MR1736958}{article}{
      author={Br{\"u}stle, T.},
      author={Hille, L.},
      author={R{\"o}hrle, G.},
      author={Zwara, G.},
       title={The {B}ruhat-{C}hevalley order of parabolic group actions in
  general linear groups and degeneration for {$\Delta$}-filtered modules},
        date={1999},
        ISSN={0001-8708},
     journal={Adv. Math.},
      volume={148},
      number={2},
       pages={203\ndash 242},
         url={http://dx.doi.org/10.1006/aima.1999.1851},
      review={\MR{1736958 (2001c:14074)}},
}

\bib{MR879187}{article}{
      author={B\"{u}rgstein, Hartmut},
      author={Hesselink, Wim~H.},
       title={Algorithmic orbit classification for some {B}orel group actions},
        date={1987},
        ISSN={0010-437X,1570-5846},
     journal={Compositio Math.},
      volume={61},
      number={1},
       pages={3\ndash 41},
         url={http://www.numdam.org/item?id=CM_1987__61_1_3_0},
      review={\MR{879187}},
}

\bib{MR2255427}{article}{
      author={Cameron, Peter},
      author={Prellberg, Thomas},
      author={Stark, Dudley},
       title={Asymptotics for incidence matrix classes},
        date={2006},
        ISSN={1077-8926},
     journal={Electron. J. Combin.},
      volume={13},
      number={1},
       pages={Research Paper 85, 19},
         url={https://doi.org/10.37236/1111},
      review={\MR{2255427}},
}

\bib{MR1994224}{article}{
      author={Cortez, Aur{\'e}lie},
       title={Singularit\'es g\'en\'eriques et quasi-r\'esolutions des
  vari\'et\'es de {S}chubert pour le groupe lin\'eaire},
        date={2003},
        ISSN={0001-8708},
     journal={Adv. Math.},
      volume={178},
      number={2},
       pages={396\ndash 445},
         url={http://dx.doi.org/10.1016/S0001-8708(02)00081-6},
      review={\MR{1994224}},
}

\bib{MR1834739}{article}{
      author={Crawley-Boevey, William},
       title={Geometry of the moment map for representations of quivers},
        date={2001},
        ISSN={0010-437X},
     journal={Compositio Math.},
      volume={126},
      number={3},
       pages={257\ndash 293},
         url={https://doi.org/10.1023/A:1017558904030},
      review={\MR{1834739}},
}

\bib{MR1944812}{article}{
      author={Crawley-Boevey, William},
      author={Schr\"{o}er, Jan},
       title={Irreducible components of varieties of modules},
        date={2002},
        ISSN={0075-4102},
     journal={J. Reine Angew. Math.},
      volume={553},
       pages={201\ndash 220},
         url={https://doi.org/10.1515/crll.2002.100},
      review={\MR{1944812}},
}

\bib{MR0332887}{article}{
      author={Gabriel, Peter},
       title={Unzerlegbare {D}arstellungen. {I}},
        date={1972},
        ISSN={0025-2611},
     journal={Manuscripta Math.},
      volume={6},
       pages={71\ndash 103; correction, ibid. 6 (1972), 309},
      review={\MR{0332887 (48 \#11212)}},
}

\bib{MR1827861}{article}{
      author={Gasharov, Vesselin},
       title={Sufficiency of {L}akshmibai-{S}andhya singularity conditions for
  {S}chubert varieties},
        date={2001},
        ISSN={0010-437X},
     journal={Compositio Math.},
      volume={126},
      number={1},
       pages={47\ndash 56},
         url={http://dx.doi.org/10.1023/A:1017585921369},
      review={\MR{1827861}},
}

\bib{MR2144987}{article}{
      author={Geiss, Christof},
      author={Leclerc, Bernard},
      author={Schr\"{o}er, Jan},
       title={Semicanonical bases and preprojective algebras},
        date={2005},
        ISSN={0012-9593},
     journal={Ann. Sci. \'{E}cole Norm. Sup. (4)},
      volume={38},
      number={2},
       pages={193\ndash 253},
         url={https://doi.org/10.1016/j.ansens.2004.12.001},
      review={\MR{2144987}},
}

\bib{MR2115084}{article}{
      author={Geiss, Christof},
      author={Schr\"{o}er, Jan},
       title={Extension-orthogonal components of preprojective varieties},
        date={2005},
        ISSN={0002-9947},
     journal={Trans. Amer. Math. Soc.},
      volume={357},
      number={5},
       pages={1953\ndash 1962},
         url={https://doi.org/10.1090/S0002-9947-04-03555-X},
      review={\MR{2115084}},
}

\bib{MR2123984}{article}{
      author={Goodwin, Simon},
       title={Algorithmic testing for dense orbits of {B}orel subgroups},
        date={2005},
        ISSN={0022-4049,1873-1376},
     journal={J. Pure Appl. Algebra},
      volume={197},
      number={1-3},
       pages={171\ndash 181},
         url={https://doi.org/10.1016/j.jpaa.2004.08.038},
      review={\MR{2123984}},
}

\bib{MR2356319}{article}{
      author={Goodwin, Simon~M.},
      author={Hille, Lutz},
       title={Prehomogeneous spaces for {B}orel subgroups of general linear
  groups},
        date={2007},
        ISSN={1083-4362,1531-586X},
     journal={Transform. Groups},
      volume={12},
      number={3},
       pages={475\ndash 498},
         url={https://doi.org/10.1007/s00031-006-0052-1},
      review={\MR{2356319}},
}

\bib{MR1669178}{article}{
      author={Hille, L.},
      author={R{\"o}hrle, G.},
       title={A classification of parabolic subgroups of classical groups with
  a finite number of orbits on the unipotent radical},
        date={1999},
        ISSN={1083-4362},
     journal={Transform. Groups},
      volume={4},
      number={1},
       pages={35\ndash 52},
         url={http://dx.doi.org/10.1007/BF01236661},
      review={\MR{1669178 (2000f:20072)}},
}

\bib{MR2033108}{article}{
      author={Hille, Lutz},
      author={Vossieck, Dieter},
       title={The quasi-hereditary algebra associated to the radical bimodule
  over a hereditary algebra},
        date={2003},
        ISSN={0010-1354},
     journal={Colloq. Math.},
      volume={98},
      number={2},
       pages={201\ndash 211},
         url={http://dx.doi.org/10.4064/cm98-2-6},
      review={\MR{2033108 (2005c:16021)}},
}

\bib{MR3228477}{article}{
      author={Jensen, Bernt~Tore},
      author={Su, Xiuping},
       title={Adjoint action of automorphism groups on radical endomorphisms,
  generic equivalence and {D}ynkin quivers},
        date={2014},
        ISSN={1386-923X,1572-9079},
     journal={Algebr. Represent. Theory},
      volume={17},
      number={4},
       pages={1095\ndash 1136},
         url={https://doi.org/10.1007/s10468-013-9436-9},
      review={\MR{3228477}},
}

\bib{MR2481983}{article}{
      author={Jensen, Bernt~Tore},
      author={Su, Xiuping},
      author={Yu, Rupert Wei~Tze},
       title={Rigid representations of a double quiver of type {$A$}, and
  {R}ichardson elements in seaweed {L}ie algebras},
        date={2009},
        ISSN={0024-6093,1469-2120},
     journal={Bull. Lond. Math. Soc.},
      volume={41},
      number={1},
       pages={1\ndash 15},
         url={https://doi.org/10.1112/blms/bdn087},
      review={\MR{2481983}},
}

\bib{MR3758148}{article}{
      author={Kang, Seok-Jin},
      author={Kashiwara, Masaki},
      author={Kim, Myungho},
      author={Oh, Se-jin},
       title={Monoidal categorification of cluster algebras},
        date={2018},
        ISSN={0894-0347},
     journal={J. Amer. Math. Soc.},
      volume={31},
      number={2},
       pages={349\ndash 426},
         url={https://doi.org/10.1090/jams/895},
      review={\MR{3758148}},
}

\bib{MR1458969}{article}{
      author={Kashiwara, Masaki},
      author={Saito, Yoshihisa},
       title={Geometric construction of crystal bases},
        date={1997},
        ISSN={0012-7094},
     journal={Duke Math. J.},
      volume={89},
      number={1},
       pages={9\ndash 36},
         url={http://dx.doi.org/10.1215/S0012-7094-97-08902-X},
      review={\MR{1458969 (99e:17025)}},
}

\bib{MR2015302}{article}{
      author={Kassel, Christian},
      author={Lascoux, Alain},
      author={Reutenauer, Christophe},
       title={The singular locus of a {S}chubert variety},
        date={2003},
        ISSN={0021-8693},
     journal={J. Algebra},
      volume={269},
      number={1},
       pages={74\ndash 108},
         url={http://dx.doi.org/10.1016/S0021-8693(03)00014-0},
      review={\MR{2015302 (2005f:14096)}},
}

\bib{MR1944442}{book}{
      author={Kimura, Tatsuo},
       title={Introduction to prehomogeneous vector spaces},
      series={Translations of Mathematical Monographs},
   publisher={American Mathematical Society, Providence, RI},
        date={2003},
      volume={215},
        ISBN={0-8218-2767-7},
        note={Translated from the 1998 Japanese original by Makoto Nagura and
  Tsuyoshi Niitani and revised by the author},
      review={\MR{1944442}},
}

\bib{MR1051089}{article}{
      author={Lakshmibai, V.},
      author={Sandhya, B.},
       title={Criterion for smoothness of {S}chubert varieties in {${\rm
  Sl}(n)/B$}},
        date={1990},
        ISSN={0253-4142},
     journal={Proc. Indian Acad. Sci. Math. Sci.},
      volume={100},
      number={1},
       pages={45\ndash 52},
         url={http://dx.doi.org/10.1007/BF02881113},
      review={\MR{1051089}},
}

\bib{MR3866895}{article}{
      author={Lapid, Erez},
      author={M\'{\i}nguez, Alberto},
       title={Geometric conditions for {$\square$}-irreducibility of certain
  representations of the general linear group over a non-archimedean local
  field},
        date={2018},
        ISSN={0001-8708},
     journal={Adv. Math.},
      volume={339},
       pages={113\ndash 190},
         url={https://doi.org/10.1016/j.aim.2018.09.027},
      review={\MR{3866895}},
}

\bib{MR4847251}{article}{
      author={Lapid, Erez},
      author={M\'{\i}nguez, Alberto},
       title={A binary operation on irreducible components of {L}usztig's
  nilpotent varieties {II}: applications and conjectures for representations of
  {${\rm GL}_n$} over a non-archimedean local field},
        date={2025},
        ISSN={1558-8599,1558-8602},
     journal={Pure Appl. Math. Q.},
      volume={21},
      number={2},
       pages={813\ndash 863},
         url={https://doi.org/10.4310/pamq.241205005734},
      review={\MR{4847251}},
}

\bib{MR1959765}{article}{
      author={Leclerc, B.},
       title={Imaginary vectors in the dual canonical basis of {$U_q(\germ
  n)$}},
        date={2003},
        ISSN={1083-4362},
     journal={Transform. Groups},
      volume={8},
      number={1},
       pages={95\ndash 104},
         url={http://dx.doi.org/10.1007/BF03326301},
      review={\MR{1959765}},
}

\bib{MR1088333}{article}{
      author={Lusztig, G.},
       title={Quivers, perverse sheaves, and quantized enveloping algebras},
        date={1991},
        ISSN={0894-0347},
     journal={J. Amer. Math. Soc.},
      volume={4},
      number={2},
       pages={365\ndash 421},
         url={http://dx.doi.org/10.2307/2939279},
      review={\MR{1088333}},
}

\bib{MR1758244}{article}{
      author={Lusztig, G.},
       title={Semicanonical bases arising from enveloping algebras},
        date={2000},
        ISSN={0001-8708},
     journal={Adv. Math.},
      volume={151},
      number={2},
       pages={129\ndash 139},
         url={http://dx.doi.org/10.1006/aima.1999.1873},
      review={\MR{1758244}},
}

\bib{MR1853139}{article}{
      author={Manivel, L.},
       title={Le lieu singulier des vari\'et\'es de {S}chubert},
        date={2001},
        ISSN={1073-7928},
     journal={Internat. Math. Res. Notices},
      number={16},
       pages={849\ndash 871},
         url={http://dx.doi.org/10.1155/S1073792801000423},
      review={\MR{1853139}},
}

\bib{MR3729270}{book}{
      author={Milne, J.~S.},
       title={Algebraic groups},
      series={Cambridge Studies in Advanced Mathematics},
   publisher={Cambridge University Press, Cambridge},
        date={2017},
      volume={170},
        ISBN={978-1-107-16748-3},
         url={https://doi.org/10.1017/9781316711736},
        note={The theory of group schemes of finite type over a field},
      review={\MR{3729270}},
}

\bib{MR0390138}{article}{
      author={Pjasecki{\u\i}, V.~S.},
       title={Linear {L}ie groups that act with a finite number of orbits},
        date={1975},
        ISSN={0374-1990},
     journal={Funkcional. Anal. i Prilo\v zen.},
      volume={9},
      number={4},
       pages={85\ndash 86},
      review={\MR{0390138 (52 \#10964)}},
}

\bib{MR130878}{article}{
      author={Rosenlicht, Maxwell},
       title={On quotient varieties and the affine embedding of certain
  homogeneous spaces},
        date={1961},
        ISSN={0002-9947},
     journal={Trans. Amer. Math. Soc.},
      volume={101},
       pages={211\ndash 223},
         url={https://doi.org/10.2307/1993371},
      review={\MR{130878}},
}

\bib{MR430336}{article}{
      author={Sato, M.},
      author={Kimura, T.},
       title={A classification of irreducible prehomogeneous vector spaces and
  their relative invariants},
        date={1977},
        ISSN={0027-7630},
     journal={Nagoya Math. J.},
      volume={65},
       pages={1\ndash 155},
         url={http://projecteuclid.org/euclid.nmj/1118796150},
      review={\MR{430336}},
}

\bib{MR2422304}{article}{
      author={Woo, Alexander},
      author={Yong, Alexander},
       title={Governing singularities of {S}chubert varieties},
        date={2008},
        ISSN={0021-8693},
     journal={J. Algebra},
      volume={320},
      number={2},
       pages={495\ndash 520},
         url={http://dx.doi.org/10.1016/j.jalgebra.2007.12.016},
      review={\MR{2422304}},
}

\end{biblist}
\end{bibdiv}
\end{document}